\documentclass[a4paper,12pt,leqno]{amsart}
\usepackage{amsmath}
\usepackage{amsfonts}
\usepackage{amssymb}
\usepackage{amsthm}
\usepackage{amscd}
\usepackage{amsgen}
\usepackage{latexsym}
\usepackage{url}

\newcommand{\CC}{{\mathcal C}}

\newcommand{\F}{{\mathcal F}}
\newcommand{\G}{{\mathcal G}}
\newcommand{\SV}{{\mathcal S}}

\newcommand{\R}{{\mathbb R}}
\newcommand{\N}{{\mathbb N}}
\newcommand{\C}{{\mathbb C}}

\newcommand{\diag}{\operatorname{diag}}
\newcommand{\tr}{\operatorname{tr}}

\def\B{{\mathcal B}}
\def\C{{\mathbb C}}
\def\P{{\mathcal P}}
\def\R{{\mathcal R}}
\def\T{{\mathcal T}}
\def\Rho{{\sf P}}
\def\J{{\sf J}}
\def\f{{\frac{n}{2}}}
\def\st{\stackrel{\text{def}}{=}}

\newtheorem{defn}{Definition}[section]

\newtheorem{conj}{Conjecture}[section]
\newtheorem{propo}{Proposition}[section]
\newtheorem{lemm}{Lemma}[section]
\newtheorem{corr}{Corollary}[section]
\newtheorem{ex}{Example}[section]

\numberwithin{equation}{section}
\numberwithin{table}{section}

\title{Universal recursive formulae for $Q$-curvatures}

\author{Carsten Falk}
\address{Humboldt-Universit\"at, Institut f\"ur Mathematik, Unter den Linden,
10099 Berlin}
\email{falk@math.hu-berlin.de}

\author{Andreas Juhl}
\address{Humboldt-Universit\"at, Institut f\"ur Mathematik, Unter den Linden,
10099 Berlin}
\email{ajuhl@math.hu-berlin.de}

\begin{document}

\begin{abstract} We formulate and discuss two conjectures concerning
recursive formulae for Branson's $Q$-curvatures. The proposed
formulae describe all $Q$-curvatures on manifolds of all even
dimensions in terms of respective lower order $Q$-curvatures and
lower order GJMS-operators. They are universal in the dimension of
the underlying space. The recursive formulae are generated by an
algorithm which rests on the theory of residue families of
\cite{juhl}. We attempt to resolve the algorithm by formulating a
conjectural description of the coefficients in the recursive
formulae in terms of interpolation polynomials associated to
compositions of natural numbers. We prove that the conjectures cover
$Q_4$ and $Q_6$ for general metrics, and $Q_8$ for conformally flat
metrics. The result for $Q_8$ is proved here for the first time.
Moreover, we display explicit (conjectural) formulae for
$Q$-curvatures of order up to $16$, and test high order cases for
round spheres and Einstein metrics.
\end{abstract}

\maketitle

\centerline \today

\tableofcontents

\renewcommand{\thefootnote}{}

\footnotetext{The work of the second author was supported by SFB 647
``Raum-Zeit-Materie'' of DFG.}

\footnotetext{MSC 2000: Primary 53B20, 53C20, Secondary 53A30,
58J50.}

\section{Introduction}\label{intro}

For any Riemannian manifold $(M,h)$ of even dimension $n$, there is
a finite sequence $P_{2N}(h)$ ($1 \le N \le \frac{n}{2}$) of natural
differential operators on functions on $M$ with leading part $\Delta_h^N$
which transform as
$$
e^{(\f+N)\varphi} \circ P_{2N}(\hat{h}) \circ e^{-(\f-N)\varphi} = P_{2N}(h)
$$
under conformal changes $\hat{h} = e^{2\varphi}h$ of the metric.
These operators were derived in \cite{GJMS} from the powers of the
Laplacian of the Fefferman-Graham ambient metric (see \cite{cartan}
and \cite{FG-final}). For $2N > n$, the construction in \cite{GJMS} is
obstructed by the Fefferman-Graham tensor. More sharply, in that range
it is impossible to construct a conformally covariant operator (for
all metrics) by adding lower order terms to $\Delta^N$ (\cite{GJMS-2},
\cite{G-H}). On the other hand, if such operators exist, they are not
uniquely determined by conformal covariance. In the following,
$P_{2N}$ will denote the operators constructed in \cite{GJMS}, and
they will be referred to as the GJMS-operators.

$P_2$ and $P_4$ are the well-known Yamabe and Paneitz operator which
are given by
\begin{align*}
P_2 & = \Delta - \left(\f\!-\!1\right) \J, \\
P_4 & = \Delta^2 + \delta( (n\!-\!2) \J \!-\! 4 \Rho ) \# d 
+ \left(\f\!-\!2\right)
\left( \f \J^2 \!-\! 2 |\Rho|^2 \!-\! \Delta \J \right),
\end{align*}
respectively. Here
$$
\Rho = \frac{1}{n\!-\!2} \left(Ric\!-\!\frac{\tau}{2(n\!-\!1)}h\right)
$$
denotes the Schouten tensor of $h$, $\tau$ denotes the scalar
curvature, and $\J = \frac{\tau}{2(n-1)}$ is the trace of $\Rho$.
$\#$ denotes the natural action of symmetric bilinear forms on
$1$-forms. Explicit expressions for the higher order operators
$P_{2N}$ for $N \ge 3$ are considerably more complicated.

The GJMS-operators $P_{2N}$ give rise to a finite sequence $Q_{2N}$
($1 \le N \le \frac{n}{2}$) of Riemannian curvature invariants
according to
\begin{equation}\label{q-def}
P_{2N}(h)(1) = (-1)^N \left(\f \!-\!N \right) Q_{2N}(h)
\end{equation}
(see \cite{comp}). $Q_{2N}$ is a curvature invariant of order $2N$,
i.e., it involves $2N$ derivatives of the metric. In the following,
the quantities $Q_{2N}(h)$ will be called the {\em $Q$-curvatures} of
$h$.

In particular, we find
\begin{equation}\label{q-2-4}
Q_2 = \J \quad \mbox{and} \quad Q_4 = \f \J^2 \!-\! 2 |\Rho|^2
\!-\!\Delta \J.
\end{equation}
Explicit formulae for $Q_{2N}$ for $N \ge 3$ are considerably more
complicated.

The critical GJMS-operator $P_n$ and the critical $Q$-curvature $Q_n$
play a special role. In that case, \eqref{q-def} does not define
$Q_n$, however. Instead, $Q_n$ arises by continuation in dimension
from the subcritical $Q$-curvatures $Q_{2N}$ ($2N<n$). The pair
$(P_n,Q_n)$ satisfies the fundamental identity
\begin{equation}\label{FI}
e^{n \varphi} Q_n(\hat{h}) = Q_n(h) + (-1)^\f P_n(h)(\varphi).
\end{equation}
It shows that the transformation of $Q_n$ under conformal changes of
$h$ is governed by the {\em linear} differential operator $P_n$. This
is one of the remarkable properties of Branson's $Q$-curvature
$Q_n$. \eqref{FI} implies that, for closed $M$, the total
$Q$-curvature
\begin{equation}\label{totalQ}
\int_M Q_n vol
\end{equation}
is a {\em global} conformal invariant.

Despite the simple formulae \eqref{q-2-4}, it remains notoriously
difficult to find good expressions for $Q$-curvatures of higher
order. Explicit formulae for $Q_6$ and $Q_8$ in arbitrary dimension
were given in \cite{G-P}. For conformally flat metrics and general
dimensions, $Q_6$ already appeared in \cite{comp}.

It is natural to expect that the complexity of the quantities $Q_{2N}$
increases exponentially with the order. This is one of the aspects in
which its behaviour resembles that of the heat coefficients of
self-adjoint elliptic differential operators. The relations between
both quantities are much more substantial, though.  The problem to
understand the structure of heat coefficients of conformally covariant
operators was actually one of the origins of the notion of
$Q$-curvature \cite{branson-book}. Explicit formulae for heat
coefficients are known only for sufficiently small orders.  There is
an extensive literature devoted to such formulae (see \cite{vasil} for
a recent review).

The lack of information concerning the structure of high order
$Q$-curvatures presently seems to obstruct the understanding of its
nature and its proper role in geometric analysis (see \cite{malch}
for a review in dimension $4$).

In the present work we propose a uniform description of all
$Q$-curvatures with the following main features.
\begin{itemize}
\item [1.] Any $Q$-curvature is the sum of two parts of different
nature.
\item [2.] The main part is a linear combination of respective
lower order GJMS-operators acting on lower order $Q$-curvatures with
coefficients which do not depend on the dimension of the underlying
space.
\item [3.] The second part is defined in terms of the constant term of
a power of the Yamabe-operator of an associated Poincar\'e-Einstein metric.
\end{itemize}
These properties motivate to refer to the proposed formulae as {\em
universal} and {\em recursive}.

In more detail, Conjecture \ref{main-c} asserts that on manifolds of even
dimension $n$,
\begin{equation}\label{GRQ}
Q_{2N} = \sum_I a_I^{(N)} P_{2I}(Q_{2N-2|I|}) + (-1)^{N-1}
\frac{(2N\!-\!2)!!}{(2N\!-\!3)!!} i^* \bar{P}_2^{N-1}(\bar{Q}_2)
\end{equation}
for all non-negative integers $N$ so that $2N \le n$. The rational
coefficients $a_I^{(N)}$ are generated by an algorithm which will be
defined in Section \ref{status}. The sum in \eqref{GRQ} runs over all
compositions $I$ of integers in $[1,N-1]$ as sums of natural
numbers. Moreover, we use the following notation. For a composition
$I=(I_1,\dots,I_m)$ of size $|I| = \sum_i I_i$, we set
$$
P_{2I} = P_{2I_1} \circ \dots \circ P_{2I_m}.
$$
In \eqref{GRQ} for the metric $h$, the operator $\bar{P}_2$ denotes
the Yamabe operator of the conformal compactification $dr^2+h_r$ of
the Poincar\'e-Einstein metric of $h$ (the relevant constructions are
reviewed in Section \ref{theory}). Similarly, $\bar{Q}_2$ is $Q_2$ for
the metric $dr^2+h_r$, and $i^*$ restricts functions to $r=0$.

Alternatively, the quantity $i^*\bar{P}_2^{N-1}(\bar{Q}_2)$ can be
written in the form
$$
-\frac{n\!-\!1}{2} i^* \bar{P}_2^N (1).
$$
However, we prefer to use the form \eqref{GRQ} which hides the
dimension $n$ of the underlying space.

The existence of recursive formulae for general $Q_{2N}$ has been an
open problem since the invention of $Q$-curvature. \eqref{GRQ}
proposes some answer.

One might also ask for recursive formulae for $Q_{2N}$ which rest
only on lower order GJMS-operators and lower order $Q$-curvatures of
the given metric. In view of the contribution $i^*
\bar{P}_2^{N-1}(\bar{Q}_2)$, the formula \eqref{GRQ} is {\em not} of
this form. However, already for $N=2$ such formulae are unlikely to
exist since $Q_4$ depends on the full Ricci tensor whereas $P_2$ and
$Q_2$ only depend on scalar curvature.

The presentations \eqref{GRQ} imply that the structure of the
constant term of any GJMS-operator is influenced by {\em all} lower
order GJMS-operators. This illustrates the enormous complexity of
the GJMS operators. The recursive structure for $Q$-curvature seems
to be a phenomenon which is not known to have analogs for related
quantities as, for instance, the heat coefficients (see
\eqref{conf-ind}).

Next, we make explicit \eqref{GRQ} for $Q_4$, $Q_6$ and $Q_8$. In
these cases, the asserted formulae are theorems and we briefly
indicate their proofs. We start with a version for $Q_2$. It just
says that
\begin{equation}\label{q2-ex}
Q_2 = i^* \bar{Q}_2
\end{equation}
in all dimensions (see \eqref{u2}). Next, the universal recursive
formula for $Q_4$ states that \label{calcul}
\begin{equation}\label{q4-ex}
Q_4 = P_2 (Q_2) - 2 i^* \bar{P}_2 (\bar{Q}_2).
\end{equation}
This formula is valid in all dimensions $n \ge 4$, i.e.,
\eqref{q4-ex} is universal. In fact, it reads
$$
Q_4 = \left(\Delta \!-\! \frac{n\!-\!2}{2}\J \right)(\J) - 2 i^*
\left( (\partial/\partial r)^2 + \Delta_{h_r} - \frac{n\!-\!1}{2}
\bar{Q}_2 \right) (\bar{Q}_2)
$$
(see Section \ref{theory} for the notation). Using $i^* \bar{Q}_2 =
Q_2 = \J$ (see \eqref{q2-ex}) and
$$
i^*(\partial/\partial r)^2 (\bar{Q}_2) = |\Rho|^2,
$$
the sum simplifies to
$$
\frac{n}{2} \J^2 - 2|\Rho|^2 -\Delta \J.
$$
This shows the equivalence of \eqref{q4-ex} and the traditional
formula \eqref{q-2-4} for $Q_4$. The presentation \eqref{q4-ex} is
distinguished by the fact that it is uniform in all dimensions. A
disadvantage of \eqref{q4-ex} is that the fundamental transformation
law \eqref{FI} in the critical dimension $n=4$ is less obvious from
this formula. In this aspects, \eqref{q4-ex} resembles the holographic
formula \eqref{holo-4}.

Next, we have the recursive formula
\begin{equation}\label{q6-ex}
Q_6 = \frac{2}{3} P_2(Q_4) + \left[-\frac{5}{3} P_2^2 + \frac{2}{3}
P_4\right](Q_2) + \frac{8}{3} i^* \bar{P}_2^2(\bar{Q}_2)
\end{equation}
for $Q_6$ in all dimensions $n \ge 6$. A detailed proof of
\eqref{q6-ex} can be found in \cite{juhl}. It is a special case of
the algorithm of Section \ref{status}.

For $n=6$, the holographic formula \eqref{holo} of \cite{gj-q}
presents $Q_6$ in the form
\begin{equation}\label{q6-bach}
Q_6 = 16 \tr (\Rho^3) - 24 \J |\Rho|^2 + 8 \J^3 + 8 (\B,\Rho) +
\mbox{divergence terms},
\end{equation}
where $\B$ denotes a version of the Bach tensor. The recursive
formula \eqref{q6-ex} covers the contribution $(\B,\Rho)$ in
\eqref{q6-bach} by the term
$$
\frac{8}{3}(\partial/\partial r)^4|_0 (\bar{Q}_2).
$$
This illustrates the role of the term which involves $\bar{P}_2$ and
$\bar{Q}_2$. An extension of this observation to the general case
will be discussed in Section \ref{status}.

We also note that \eqref{q6-ex} is equivalent to a formula of Gover
and Peterson \cite{G-P}. For a proof of this fact we refer to \cite{juhl}.

We continue with the description of the recursive formula for $Q_8$.
In the critical dimension $n=8$, the algorithm of Section
\ref{status} yields
\begin{multline}\label{q8-ex}
Q_8 = \frac{3}{5} P_2(Q_6)
+ \left[ - 4 P_2^2 + \frac{17}{5} P_4 \right](Q_4) \\
+ \left[ -\frac{22}{5} P_2^3 + \frac{8}{5} P_2 P_4 + \frac{28}{5}
P_4 P_2 - \frac{9}{5} P_6 \right](Q_2) - \frac{16}{5} i^*
\bar{P}_2^3(\bar{Q}_2)
\end{multline}
for locally conformally flat metrics (Proposition \ref{flat-crit-8}).
Using a second algorithm, we prove that \eqref{q8-ex} holds true in
all dimensions $n \ge 8$ (Proposition \ref{flat-8}). It remains open,
whether \eqref{q8-ex} extends to general metrics. The relation between
\eqref{q8-ex} and the Gover-Peterson formula \cite{G-P} for $Q_8$ is
not yet understood.

For $N \ge 5$, Conjecture \ref{main-c} enters largely unexplored
territory. We outline the algorithm which generates the
presentations \eqref{GRQ}. First, we generate such a presentation
for the critical $Q$-curvature $Q_n$. For this, we apply an
algorithm which rests on the relation of the critical $Q$-curvature
$Q_n$ to the quantity
$$
\dot{D}_n^{res}(0)(1)
$$
and the recursive structure of all residue families
$D_{2N}^{res}(\lambda)$ for $2N \le n$. We refer to Section
\ref{theory} for the definition of the relevant concepts. The details
of the algorithm are explained in Section \ref{status}. An important
argument which enters into the algorithm is the principle of {\em
  universality}. It plays the following role. The algorithm for $Q_n$
uses the assumption that the analogously generated presentations of
{\em all} lower order $Q$-curvatures $Q_{2N}$, $N=1,\dots,\f-1$ hold
true on manifolds of dimension $n$. In particular, the derivation of
\eqref{q6-ex} in dimension $n=6$, uses the facts that \eqref{q2-ex}
and \eqref{q4-ex} hold true in dimension $n=6$. Similarly, the
derivation of \eqref{q8-ex} in dimension $n=8$ applies the facts that
the formulae \eqref{q2-ex}, \eqref{q4-ex} and \eqref{q6-ex} hold true
in $n=8$. Under the assumption of universality, the algorithm
generates a formula for $Q_n$. Since universality is open, the
identification of the resulting formula with $Q_n$ is only
conjectural. Conjecture \ref{main-c} asserts that the resulting
formula for $Q_n$ again is universal, i.e., holds true in all
dimensions $>n$. In order to apply the factorization identities of
residue families we restrict to conformally flat metrics. In low order
cases, this restriction can be removed. It hopefully is superfluous in
general.

With these motivations, it becomes important to describe the structure
of the right-hand sides of \eqref{GRQ} generated by the above
algorithm. Although the algorithm only involves linear algebra, the
complexity of calculations quickly increases with $N$. In particular,
we were unable to find closed formulae for the coefficients
$a_I^{(N)}$.

Instead, we describe an attempt to resolve the algorithm by relating
it to another much simpler algorithm which deals with polynomials
instead of operators. More precisely, we introduce an algorithm for
the generation of a system of polynomials. It associates a canonical
polynomial $r_I$ to any composition $I$. The degree of the polynomial
$r_I$ is $2|I|-1$. Conjecture \ref{para} relates, for any $I$, the
restriction of $r_I$ to $\N$ to the function $N \mapsto
a_I^{(N)}$. The formulation of this conjecture results from an
analysis of computer assisted calculations of the coefficients
$a_I^{(N)}$. In particular, such calculations indicate that the
functions $N \to a_I^{(N)}$ can be described by interpolation
polynomials. A deeper analysis of the numerical data leads to a
description of these polynomials in terms of other interpolation
problems.

We describe the content of Conjecture \ref{para} for the
coefficients of
$$
P_{2k}(Q_{2N-2k}), \; N \ge k+1
$$
and
$$
P_{2j} P_{2k} (Q_{2N-2j-2k}), \; N \ge j+k+1.
$$

For $k \ge 1$, let $r_{(k)}$ be the unique polynomial of degree
$2k-1$ which is characterized by its $2k$ values
\begin{equation*}
r_{(k)}(-i) = 0, \quad i = 1,\dots,k-1,
\end{equation*}
and
$$
r_{(k)} \left(\frac{1}{2}-i\right) = (-2)^{-(k-1)}
\frac{\left(\frac{1}{2}\right)_{k-1}}{(k\!-\!1)!}, \quad i =
0,1,\dots,k.
$$
The second set of conditions can be replaced by the simpler
requirement that $r_{(k)}$ is constant on the set
$$
\SV(k) = \left\{\frac{1}{2}-k,\dots,-\frac{1}{2},\frac{1}{2}\right\}
$$
together with the condition that
$$
r_{(k)}(0) = (-1)^{k-1} \frac{(2k\!-\!3)!!}{k!}.
$$
Now Conjecture \ref{para} says that
\begin{equation}\label{a-GJMS}
a_{(k)}^{(N)} = \prod_{i=1}^k \left(\frac{N\!-\!i}{2N\!-\!2i\!-\!1}
\right) r_{(k)}(N\!-\!k), \; N \ge k+1.
\end{equation}

For a composition $I=(j,k)$ with two entries, we define a unique
polynomial $r_{(j,k)}$ of degree $2j+2k-1$ by the $j+k-1$ conditions
\begin{equation}\label{two-1}
r_{(j,k)}(-i) = 0, \quad i=1, \dots, j + k, \; i \ne k,
\end{equation}
the $j+k+1$ conditions
\begin{equation}\label{two-2}
r_{(j,k)}(\cdot) + r_{(j)} \left( \frac{1}{2} \right) r_{(k)}(\cdot)
= r_{(j,k)} \left( \frac{1}{2} \right) + r_{(j)}
\left(\frac{1}{2}\right) r_{(k)} \left(\frac{1}{2}\right)
\end{equation}
on the set
$$
\SV(j+k)= \left\{ \frac{1}{2},\frac{1}{2}-1,\dots,\frac{1}{2}-(j+k)
\right\},
$$
and the relation
\begin{equation}\label{two-3}
r_{(j,k)}(0) = - r_{(j)}(k) r_{(k)}(0).
\end{equation}
\eqref{two-2} can be replaced by the simpler condition that the left
hand side is constant on the set $\SV(j+k)$. The value of that
constant is determined by the additional relation \eqref{two-3} for
the constant term of $r_{(j,k)}$. Now Conjecture \ref{para} says
that
\begin{equation}\label{2-GJMS}
a_{(j,k)}^{(N)} = \prod_{i=1}^{j+k}
\left(\frac{N-i}{2N\!-\!2i\!-\!1}\right) r_{(j,k)}(N\!-\!(j\!+\!k)),
\; N \ge j+k+1.
\end{equation}

For general compositions $I$, there are analogous interpolation
polynomials $r_I$. However, the interpolation data are more
complicated. Indeed, those for $r_I$ are recursively determined by
those of polynomials $r_J$ which are associated to sub-compositions
$J$ of $I$. The corresponding recursive relations are non-linear
(see \eqref{mult-2}, \eqref{CC-const}). By iteration, they can be
used to generate $r_I$ from the polynomials $r_{(k)}$, where $k$
runs through the entries of $I$. For the details we refer to Section
\ref{structure}.

We finish the present section with a number of comments. Branson
introduced the quantity $Q_n$ in order to systematize the study of
extremal properties of functional determinants of the Yamabe operator
$P_2$ (and other conformally covariant differential operators). The
central idea is to decompose the conformal anomalies of the
determinants as sums of a universal part (given by $Q$-curvature),
locally conformally invariant parts (which vanish in the conformally
flat case) and divergence parts with {\em local} conformal primitives
(\cite{branson-book}, \cite{comp}, \cite{br-last},
\cite{bran-MW},\cite{conf-s}). The concept rests on the observation
that the heat coefficients of conformally covariant differential
operators display similar conformal variational formulae as the
$Q$-curvatures $Q_{2j}$. We briefly describe that analogy in the case
of the Yamabe operator $D=-P_2$. Assume that $D$ is positive. The
coefficients $a_j$ in the asymptotics
$$
\tr (e^{-tD}) \sim \sum_{j\ge 0} t^{\frac{-n+j}{2}} \int_M a_j vol, \; t \to 0
$$
of the trace of its heat kernel are Riemannian curvature invariants
which satisfy the conformal variational formulae
\begin{equation}\label{heat}
\left(\int_M a_j vol \right)^\bullet [\varphi]
= (n\!-\!j) \int_M \varphi a_j vol, \; \varphi \in C^\infty(M).
\end{equation}
Here the notation $^\bullet$ is used to indicate the infinitesimal
conformal variation
$$
\F^\bullet(h)[\varphi] = (d/dt)|_0 \F(e^{2t\varphi}h)
$$
of the functional $\F$. In particular, the integral
\begin{equation}\label{conf-ind}
\int_M a_n vol
\end{equation}
is a global conformal invariant. The conformal variational formula
$$
-(\log \det (D))^\bullet [\varphi] = 2 \int_M \varphi a_n vol
$$
shows the significance of $a_n$ as a conformal anomaly of the
determinant. For the details we refer to \cite{BrO-1}, \cite{BrO-3}.

The conformal invariance of \eqref{conf-ind} has strong
implications. In fact, when combined with the Deser-Schwimmer
classification of conformal anomalies (proved by Alexakis in the
fundamental work \cite{alex-final}), it implies that $a_n$ is a linear
combination of the Pfaffian, a local conformal invariant and a
divergence. The existence of such a decomposition also follows for the
global conformal invariant \eqref{totalQ}. The conformal invariance of
\eqref{totalQ} is a consequence of
$$
\left(\int_M Q_{2j} vol \right)^\bullet [\varphi] = (n\!-\!2j) \int_M
\varphi Q_{2j} vol.
$$
The problem to find explicit versions of these decompositions is
more difficult.

A third series of related scalar curvature quantities, which in recent
years naturally appeared in connection with ideas around the
AdS/CFT-correspondence, are the holographic coefficients $v_{2j}$.
These quantities describe the asymptotics of the volume form of
Poincar\'e-Einstein metrics (Section \ref{theory}). Here
\cite{chang-fang}
\begin{align*}
\left(\int_M v_{2j} vol \right)^\bullet [\varphi] & = (n\!-\!2j)
\int_M \varphi v_{2j} vol,
\end{align*}
and the integral
\begin{equation}\label{invariant}
\int_M v_n vol
\end{equation}
is a global conformal invariant \cite{GV}. $v_n$ is the conformal
anomaly of the renormalized volume of conformally compact Einstein
metrics (\cite{GV}). The problem to understand the parallel between
renormalized volumes and functional determinants is at the center of
the AdS/CFT-duality (\cite{ho-f}, \cite{HS}).

Graham and Zworski \cite{GZ} discovered that the global conformal
invariants \eqref{invariant} and \eqref{totalQ} are proportional.
Moreover, the formula (\cite{gj-q}, \cite{juhl})
\begin{equation}\label{holo}
2n c_\f Q_n = n v_n + \sum_{j=1}^{\f-1} (n\!-\!2j) \T_{2j}^*(0) (v_{n-2j})
\end{equation}
(with $c_\f = (-1)^\f \left[2^n (\f)! (\f\!-\!1)!\right]^{-1}$) for
the critical $Q$-curvature completely expresses $Q_n$ in terms of
holographic data, $v_{2j}$ and $\T_{2j}(0)$, of the given metric. For
the definition of the differential operators $\T_{2j}(0)$ we refer to
Section \ref{theory}.

In dimension $n=4$, \eqref{holo} states that
\begin{equation}\label{holo-4}
Q_4 = 16 v_4 + 2 \Delta v_2.
\end{equation}
Using $v_4 = \frac{1}{8} (\J^2-|\Rho|^2)$ and $v_2 = -\frac{1}{2} \J$,
this is equivalent to \eqref{q-2-4}.

\eqref{holo} implies that in the conformally flat case the Pfaffian
appears naturally in $Q_n$ (as predicted by the Deser-Schwimmer
classification). Although in that case all holographic coefficients
$v_{2j}$ are known, $Q_n$ is still very complex. The complexity is
hidden in the differential operators $\T_{2j}(0)$ which define the
divergence terms. \eqref{GRQ} would shed new light on these divergence
terms by replacing the coefficients $v_{2j}$ by $Q_{2j}$, and
$\T_{2j}^*(0)$ by sums of compositions of GJMS-operators.

Finally, we note that the coefficients $v_{2j}$ for $2j \ne n$ give
rise to interesting variational problems \cite{chang-fang}. In the
conformally flat case, $v_{2j}$ is proportional to $\tr (\wedge^j
\Rho)$, and the functionals $\int_M \tr(\wedge^j \Rho) vol$ were first
studied by Viaclovski in \cite{viac}. The variational nature of the
functionals $\int_M \tr (\wedge^j \Rho)$ has been clarified by Branson
and Gover in \cite{BG-var}. For a deeper study of the quantities
$v_{2j}$ see \cite{G-ext}.

The paper is organized as follows. In Section \ref{theory}, we
describe the theoretical background from \cite{juhl}. In Section
\ref{status}, we formulate the universal recursive formula in full
generality. We combine the detailed description of the algorithm
with a clear accentuation of the conjectural input. For locally
conformally flat metrics, we prove the universality of \eqref{q8-ex}
and the recursive formula for the critical $Q_{10}$. We describe a
part of the structure of the recursive formulae in terms of a
generating function $\G$. Finally, we discuss a piece of evidence
which comes from the theory of extended obstruction tensors
\cite{G-ext}. In Section \ref{structure}, we formulate a conjectural
description of the functions $N \mapsto a_I^{(N)}$ in terms of
interpolation polynomials $r_I$ which are generated by recursive
relations (Conjecture \ref{para}). All formulated structural
properties are obtained by extrapolation from numerical data
(Section \ref{app}). The general picture is described in Section
\ref{gen-struc}. Section \ref{examp} serves as an illustration. In
particular, we reproduce all coefficients in the universal recursive
formulae for $Q_{2N}$ ($N \le 5$) in terms of the values of the
polynomials $r_I$. In Section \ref{open}, we emphasize some of the
open problems raised by the approach. In the Appendix, we display
explicit versions of the universal recursive formulae for $Q_{10}$,
$Q_{12}$, $Q_{14}$ and $Q_{16}$, test the universality of these
expressions by evaluation on round spheres of any even dimension,
and list a part of the numerical data from which the conjectures
have been distilled.

The present paper combines theoretical results of \cite{juhl} with
computer experiments using Mathematica with the NCAlgebra package.
The computer allowed to enter the almost unexplored world of
$Q$-curvatures of order exceeding $8$. The transformations of a large
number of algorithms into effective programs is the work of the first
named author.

\section{The recursive structure of residue families}\label{theory}

The algorithm which generates the proposed recursive formulae for all
$Q$-curvatures rests on two central facts. One of these is the
identity
\begin{equation}\label{hol-form}
Q_n(h) = -(-1)^\f (d/d\lambda)|_0 (D_n^{res}(h;\lambda)(1))
\end{equation}
(\cite{gj-q}, \cite{juhl}) which detects the critical $Q$-curvature
$Q_n(h)$ in the linear part of the critical residue family
$D_n^{res}(h;\lambda)$. The second fact is the recursive structure
of residue families. We start by recalling the construction of
residue families $D_{2N}^{res}(h;\lambda)$ and reviewing their basic
properties \cite{juhl}. The algorithm will be described in Section
\ref{status}.

For $2N \le n$, the families $D_{2N}^{res}(h;\lambda)$, $\lambda \in
\C$ are natural one-parameter family of local operators
$$
C^\infty([0,\varepsilon) \times M) \to C^\infty(M).
$$
They are completely determined by the metric $h$. Their construction
rests on the Poincar\'e-Einstein metrics with conformal infinity $[h]$
(\cite{cartan}, \cite{FG-final}).

A Poincar\'e-Einstein metric $g$ associated to $(M,h)$ is a metric on
$(0,\varepsilon) \times M$ (for sufficiently small $\varepsilon$) of
the form
$$
g = r^{-2} (dr^2 + h_r),
$$
where $h_r$ is a one-parameter family of metrics on $M$ so that $h_0=h$ and
\begin{equation}\label{ein}
Ric(g) + n g = O(r^{n-2}).
\end{equation}
The Taylor series of $h_r$ is even in $r$ up to order $n$. More precisely,
\begin{equation}\label{GF}
h_r = h_{(0)} + r^2 h_{(2)} + \dots + r^n(h_{(n)} + \log r \bar{h}_{(n)})
+ \dots.
\end{equation}
In \eqref{GF}, the coefficients $h_{(2)}, \dots h_{(n-2)}$ and
$\tr(h_{(n)})$ are determined by $h_{(0)} = h_0 = h$. These data are
given by polynomial formulae in terms of $h$, its inverse, and
covariant derivatives of the curvature tensor. In particular, $h_{(2)}
= -\Rho$. Let
$$
v(r,\cdot) = \frac{vol(h_r)}{vol(h)} = v_0 + r^2 v_2 + \dots + r^n v_n
+ \cdots, \; v_0 = 1.
$$
Here $vol$ refers to the volume forms of the respective metrics on
$M$. The coefficients $v_{2j} \in C^\infty(M)$ ($j=0,\dots,\f$) are
given by local formulae in terms of $h$, its inverse, and the
covariant derivatives of the curvature tensor. $v_n$ is the
holographic anomaly of the asymptotic volume of the
Poincar\'{e}-Einstein metric $g$ \cite{GV}.

\begin{defn}[{\bf Residue families}]\label{RF} For $2N \le n$, let
$$
D_{2N}^{res}(h;\lambda): C^\infty([0,\varepsilon) \times M^n) \to C^\infty(M^n)
$$
be defined by
$$
D_{2N}^{res}(h;\lambda) = 2^{2N} N! \left[(-\f\!-\!\lambda\!+\!2N\!-\!1)
\cdots (-\f\!-\!\lambda\!+\!N)\right] \delta_{2N}(h;\lambda\!+\!n\!-\!2N)
$$
with
$$
\delta_{2N}(h;\lambda) = \sum_{j=0}^{N} \frac{1}{(2N\!-\!2j)!}
\left[ \T^*_{2j}(h;\lambda) v_0 + \cdots + \T^*_0(h;\lambda) v_{2j} \right]
i^* \left(\partial/\partial r\right)^{2N-2j}.
$$
Here $i^*$ restricts functions to $r=0$, and the holographic
coefficients $v_{2j}$ act as multiplication operators.
\end{defn}

The rational families $\T_{2j}(h;\lambda)$ of differential operators
on $M$ arise by solving the asymptotic eigenfunction problem for the
Poincar\'e-Einstein metric. In other words, $\T_{2j}(h;\lambda)$ is
given by
$$
\T_{2j}(h;\lambda) f = b_{2j}(h;\lambda),
$$
where
\begin{equation}\label{asymp}
u \sim \sum_{j \ge 0} r^{\lambda+2j} b_{2j}(h;\lambda), \; r \to 0
\end{equation}
describes the asymptotics of an eigenfunction $u$ so that
$$
-\Delta_g u = \lambda(n\!-\!\lambda) u
$$
and $b_0 = f$. In particular, the operators $\T_{2j}(h;0)$ describe
the asymptotics of solutions of the Dirichlet problem at infinity.
Note that the asymptotics of an eigenfunction $u$ for $\Re(\lambda) =
\frac{n}{2}$ contains a second sum with leading exponent
$n-\lambda$. This sum is suppressed in \eqref{asymp}. The renormalized
families
$$
P_{2j}(h;\lambda) = 2^{2j} j! \left(\f\!-\!\lambda\!-\!1\right)
\cdots \left(\f\!-\!\lambda\!-\!j\right) \T_{2j}(h;\lambda)
$$
are polynomial in $\lambda$. They satisfy $P_{2j}(\lambda) = \Delta^j
+ \text{LOT}$ for all $\lambda$ and
$$
P_{2j}\left(h;\f\!-\!j\right) = P_{2j}(h).
$$
Formal adjoints of $\T_{2j}(h;\lambda)$ are taken with respect to the
scalar product defined by $h$.

The family $D_{2N}^{res}(h;\lambda)$ is conformally covariant in the
following sense. The Poincar\'e-Einstein metrics of $h$ and $\hat{h} =
e^{2\varphi} h$ are related by
$$
\kappa^* \left(r^{-2}(dr^2 \!+\!h_r)\right) = r^{-2} (dr^2\!+\!\hat{h}_r),
$$
where $\kappa$ is a diffeomorphism which fixes the boundary
$r=0$. Then we have
\begin{equation}\label{covar}
D_{2N}^{res}(\hat{h};\lambda)
= e^{(\lambda-2N) \varphi} \circ D_{2N}^{res}(h;\lambda) \circ \kappa_*
\circ \left( \frac{\kappa^*(r)}{r} \right)^\lambda.
\end{equation}
For the proof of \eqref{covar} one interprets the family as a
residue of a certain meromorphic family of distributions
\cite{juhl}.

Now assume that $h$ is conformally flat. Then for
$$
\lambda \in \left\{ -\f\!+\!N, \dots, -\f\!+\!2N\!-\!1 \right\}
\cup \left\{ -\frac{n\!-\!1}{2} \right\},
$$
the family $D_{2N}^{res}(h;\lambda)$ factorizes into the product of
a lower order residue family and a GJMS-operator:
\begin{equation}\label{fact-a}
D_{2N}^{res}\left(h;-\f\!+\!2N\!-\!j\right) = P_{2j}(h)
D_{2N-2j}^{res}\left( h;-\f\!+\!2N\!-\!j \right)
\end{equation}
for $j=1,\dots,N$ and
\begin{equation}\label{fact-b}
D_{2N}^{res}\left(h;-\frac{n\!-\!1}{2}\right) =
D_{2N-2}^{res}\left(h;-\frac{n\!+\!3}{2}\right) P_{2}(dr^2\!+\!h_r).
\end{equation}
The additional factorization identities which involve higher order
GJMS-operators for $dr^2+h_r$ (see \cite{juhl}) will not be important
in the present paper. The factorization identities should be regarded
as curved versions of multiplicity one theorems in representation
theory.

For $j=N$, \eqref{fact-a} states that
$$
D_{2N}^{res}\left(h;-\f\!+\!N\right) = P_{2N}(h) i^*.
$$
In particular, the critical residue family $D_n^{res}(h;\lambda)$
specializes to the critical GJMS-operators at $\lambda=0$:
$$
D_n^{res}(h;0) = P_n(h) i^*.
$$

The factorization identities in \eqref{fact-a} and the identity
\eqref{fact-b} are of different nature. The identities in
\eqref{fact-a} actually hold true without additional assumptions on
$h$. In \cite{juhl} it is shown that this can be derived as a
consequence of the identification of $P_{2N}$ as the residue of the
scattering operator \cite{GZ}. \eqref{fact-b} is more difficult and
presently only known for general order under the assumption that $h$
is conformally flat. In that case, the identity follows from the
conformal covariance \eqref{covar} of the family, together with a
corresponding factorization in the flat case.

\section{The universal recursive formulae}\label{status}

In the present section, we formulate conjectural recursive
presentations of all $Q$-curvatures and describe their status.

\begin{conj}[{\bf Universal recursive formulae}]\label{main-c} Let
$n$ be even and assume that $2N \le n$. Then the $Q$-curvature
$Q_{2N}$ on Riemannian manifolds of dimension $n$ can be written in
the form
\begin{equation}\label{URF}
Q_{2N} = \sum_{1 \le |I| \le N-1} a^{(N)}_I  P_{2I} (Q_{2N-2|I|}) +
(-1)^{N-1} \frac{(2N\!-\!2)!!}{(2N\!-\!3)!!} i^* \bar{P}_2^{N-1}
(\bar{Q}_2)
\end{equation}
with certain rational coefficients $a_I^{(N)}$ which do not depend
on $n$. The sum in \eqref{URF} runs over all compositions
$I=(I_1,\dots,I_m)$ of integers in $[1,N-1]$ as sums of natural
numbers. For $I=(I_1,\dots,I_m)$ of length $m$ and size $|I| = I_1 +
\dots + I_m$, the operator $P_{2I}$ is defined as the composition
$P_{2I_1} \cdots P_{2I_m}$ of GJMS-operators. The coefficients
$a_I^{(N)}$ have the sign $(-1)^{|I|+m-1}$.
\end{conj}

We emphasize that the sum in \eqref{URF} runs over compositions $I$
instead of partitions. This reflects the fact that the
GJMS-operators do not commute. Since there are $2^{N-1}$
compositions of size $N$, the sum in \eqref{URF} contains
$$
2^0 + 2^1 + \cdots + 2^{N-2} = 2^{N-1}-1
$$
terms. The operator $\bar{P}_2(h)$ denotes the Yamabe operator of
the conformal compactification $dr^2+h_r$ of the Poincar\'e-Einstein
metric of $h$ (Section \ref{theory}). $\bar{Q}_2$ is $Q_2$ for the
metric $dr^2+h_r$. In more explicit terms,
\begin{equation}\label{scalar}
\bar{Q}_2(h) = \J(dr^2\!+\!h_r) = -\frac{1}{2r} \tr (h_r^{-1} \dot{h}_r)
\end{equation}
and $\bar{P}_2(h) = \Delta_{dr^2+h_r}\!-\!(\f\!-\!1) \bar{Q}_2(h)$
with
$$
\Delta_{dr^2+h_r} = \partial^2/\partial r^2 + \frac{1}{2}
\tr(h_r^{-1} \dot{h}_r) \partial/\partial r + \Delta_{h_r}.
$$
Note that $h_{(2)} = -\Rho$ implies
\begin{equation}\label{u2}
i^* \bar{Q}_2 = Q_2.
\end{equation}

We continue with the description of the algorithm which generates
the presentations \eqref{URF}.

First of all, all formulae arise from the corresponding formulae for
critical $Q$-curvatures by applying the principle of {\em
  universality}. The conjectural status of the formulae \eqref{URF} is
partly due to the unproven applicability of this principle.

As a preparation for the definition of the algorithm, we observe
some consequences of the factorization identities for residue
families. The family $D_{2N}^{res}(h;\lambda)$ is polynomial of
degree $N$. The $N+1$ identities \eqref{fact-a} and \eqref{fact-b}
imply that $D_{2N}^{res}(h;\lambda)$ can be written as a linear
combination of the right-hand sides of these identities. The lower
order residue families which appear in this presentation, in turn,
satisfy corresponding systems of factorization identities. These
allow to write any of these families as a linear combination of the
corresponding right-hand sides of the factorization relations they
satisfy. The continuation of that process leads to a formula for
$D_{2N}^{res}(h;\lambda)$ as a linear combination of compositions of
the GJMS-operators
$$
P_{2N}(h), \dots, P_2(h)
$$
and the Yamabe operator $\bar{P}_2(h) = P_{2}(dr^2\!+\!h_r)$. The
second reason for the conjectural status of \eqref{URF} is that the
full system of factorization identities is not yet available for
general metrics (see the comments at the end of Section
\ref{theory}).

We apply the above method to the critical residue family
$D_n^{res}(h;\lambda)$ and combine the resulting formula with
\eqref{hol-form}. This yields a formula for $Q_n(h)$ as a linear
combination of compositions of the GJMS-operators
$$
P_{n-2}(h), \dots, P_2(h)
$$
and the Yamabe operator $\bar{P}_2(h) = P_{2}(dr^2+h_r)$ (acting on
$u=1$). That formula contains compositions of GJMS-operators with
powers of $\bar{P}_2(h)$ up to $\f$.

In the next step, we replace all quantities
$$
i^* \bar{P}_2^k(h)(1) \quad \mbox{for $k=1,\dots, n/2-1$}
$$
by subcritical GJMS-operators and subcritical $Q$-curvatures
$Q_{2k}$. For that purpose, we apply similar formulae for the
subcritical $Q$-curvatures. Here the principle of universality
becomes crucial. In fact, by {\em assuming} the universality of the
respective formulae for $Q_2, \dots, Q_{n-2}$, we regard these as
formulae for $i^* \bar{P}_2^k (1)$ ($1 \le k \le \f-1$), and plug
them into the formula for $Q_n$. This finishes the algorithm.

The description shows that, for conformally flat metrics, the
conjectural status of the presentations is {\em only} due to the
principle of universality.

For the convenience of the reader, we illustrate the algorithm in
two special cases.

We start with a proof of \eqref{q4-ex} in dimension $n=4$. We
consider the critical family $D_4^{res}(h;\lambda)$. We write this
family in the form
$$
A \lambda^2 + B \lambda + C,
$$
and determine the operator coefficients by using the factorization
identities
\begin{align*}
D_4^{res}(h;0) & = P_4(h) i^*, \\
D_4^{res}(h;1) & = P_2(h) D_2^{res}(h;1), \\
D_4^{res}\left(h;-\frac{3}{2}\right) & =
D_2^{res}\left(h;-\frac{7}{2}\right) P_2(dr^2\!+\!h_r).
\end{align*}
The first identity implies $C=P_4(h)i^*$. The remaining two
relations yield
$$
\begin{pmatrix} A \\[1mm] B  \end{pmatrix} =
\frac{1}{15} \begin{pmatrix} 4 & 6 \\[1mm] -4 & 9 \end{pmatrix}
\begin{pmatrix}
D_2^{res}(h;-\frac{7}{2}) P_2(dr^2\!+\!h_r) \\[1mm]
P_2(h) D_2^{res}(h;1) \end{pmatrix}.
$$
Now by the factorization identities for $D_2^{res}(h;\lambda)$,
\begin{align*}
D_2^{res}\left(h;-\frac{7}{2}\right)
& = 5 i^* P_2(dr^2 \!+\! h_r) - 4 P_2(h) i^*, \\
D_2^{res}(h;1) & = -4 i^* P_2(dr^2 \!+\! h_r) + 5 P_2(h) i^*.
\end{align*}
Thus, we find
\begin{equation}\label{n4}
A = 2 P_2^2 i^* - \frac{8}{3} P_2 i^* \bar{P}_2 + \frac{4}{3}
i^*\bar{P}_2^2 \quad \mbox{and} \quad B = 3 P_2^2 i^* - \frac{4}{3}
P_2 i^* \bar{P}_2 - \frac{4}{3} i^* \bar{P}_2^2.
\end{equation}
Now the formula for $B$ in \eqref{n4}, together with
\eqref{hol-form}, implies
\begin{align*}
Q_4 = - B(1) & = - 3 P_2^2 (1) + \frac{4}{3} P_2 (i^* \bar{P}_2(1))
+ \frac{4}{3} i^* \bar{P}_2^2 (1) \\ & = 3 P_2(Q_2) - 2 P_2 (i^*
\bar{Q}_2) - 2 i^* \bar{P}_2 (\bar{Q}_2).
\end{align*}
The last equality is a consequence of
$$
P_2(1) = - Q_2 \quad \mbox{and} \quad \bar{P}_2(1) = - \frac{3}{2}
\bar{Q}_2
$$
(see \eqref{q-def}). But using $i^* \bar{Q}_2 = Q_2$ (see
\eqref{u2}), we find
$$
Q_4 = P_2(Q_2) - 2 i^* \bar{P}_2(\bar{Q}_2).
$$
This is \eqref{q4-ex}. Although, the above derivation is only valid
in dimension $n=4$, the final formula for $Q_4$ is valid in all
dimensions (see the discussion on page \pageref{calcul}). We also
note that we simplified the contribution
$$
P_2 (i^* \bar{P}_2(1))
$$
by using $i^*\bar{Q}_2 = Q_2$ in dimension $n=4$ (see \eqref{u2}).
Since the latter identity can be regarded as a version of the
universal formula for $Q_2$, that argument is the simplest special
case of the application of universality of subcritical
$Q$-curvatures in the algorithm.

Similarly, the algorithm yields the recursive formula \eqref{q6-ex}
for the critical $Q$-curvature $Q_6$ for conformally flat metrics
$h$. The derivation makes use of the relations $i^* \bar{Q}_2 = Q_2$
and \eqref{q4-ex} in dimension $n=6$. Again, \eqref{q6-ex} holds
true for all metrics and in all dimensions $n \ge 6$. For detailed
proofs of these results we refer to \cite{juhl}. A calculation using
\eqref{q6-ex} shows that $\J^3$ contributes to $Q_6$ with the
coefficient $(\f-1)(\f+1)$.

Starting with $Q_8$, the theory is less complete. The following
detailed description of this case will also point to the open
problems. In this case, we use the universality of $i^*\bar{Q}_2 =
Q_2$, \eqref{q4-ex} and \eqref{q6-ex} to deduce the formula
\eqref{q8-ex} for $Q_8$ in dimension $n=8$ for conformally flat $h$.
The starting point is the identity
\begin{equation}\label{start-8}
-\dot{D}^{res}_8(h,0)(1) = Q_8(h).
\end{equation}
The critical family $D_8^{res}(h;\lambda)$ satisfies the
factorization identities
\begin{align*}
D_8^{res}(h;0) & = P_8(h) i^*, \\
D_8^{res}(h;1) & = P_6(h) D_2^{res}(h;1), \\
D_8^{res}(h;2) & = P_4(h) D_4^{res}(h;2), \\
D_8^{res}(h;3) & = P_2(h) D_6^{res}(h;3),
\end{align*}
and
\begin{equation}\label{last-8}
D_8^{res}\left(h;-\frac{7}{2}\right) =
D_6^{res}\left(h;-\frac{11}{2}\right) \bar{P}_2(h).
\end{equation}
In view of $P_8(h)(1) = 0$, it follows that $Q_8(h)$ can be written
as a linear combination of the four terms
$$
P_6(h) D_2^{res}(h;1)(1), \; P_4(h) D_4^{res}(h;2)(1), \; P_2(h)
D_6^{res}(h;3)(1)
$$
and $D_6^{res}(h;-\frac{11}{2}) \bar{P}_2(h)(1)$. The families
$D_{2j}^{res}(h;\lambda)$ ($j=1,2,3$), in turn, can be written as
linear combinations of compositions of respective lower order
GJMS-operators and residue families. In order to obtain these
presentations, we use the corresponding systems of factorization
identities which are satisfied by these families. The continuation
of the process leads to a presentation of $Q_8(h)$ as a linear
combination of compositions of GJMS-operators with powers of
$\bar{P}_2(h)$ (acting on $1$). More precisely, the contributions
which involve a non-trivial power of $\bar{P}_2$ are of the form
$$
* \, (i^* \bar{P}_2^k(h)(1)) \quad \mbox{for $k=1,\dots,4$}.
$$
Now we apply the universality of $i^*\bar{Q}_2 = Q_2$, \eqref{q4-ex}
and \eqref{q6-ex}. In particular, in dimension $n=8$ we regard these
formulae as expressions for
$$
i^* \bar{P}_2(h)(1), \quad i^* \bar{P}_2^2(h)(1) \quad \mbox{and}
\quad i^* \bar{P}_2^3 (h)(1)
$$
by using $\bar{P}_2 (h)(1) = -\frac{7}{2} \bar{Q}_2(h)$. These
calculations prove

\begin{propo}\label{flat-crit-8} On locally conformally flat Riemannian
manifolds of dimension $8$, $Q_8$ is given by \eqref{q8-ex}.
\end{propo}

It remains open whether, in dimension $n=8$, the same formula yields
$Q_8$ for {\em general} metrics. In the above proof, the restriction
to conformally flat metrics is only due to the unproven validity of
the factorization identity \eqref{last-8} for general metrics. We
expect that the restriction can be removed.

However, more can be said in the locally conformally flat case. In
this case, Proposition \ref{flat-8} yields the universality of
\eqref{q8-ex}. Before we prove this result, we describe a
consequence.

The validity of \eqref{q8-ex} in dimension $n=10$ (for locally
conformally flat metrics) is the only new ingredient which is
required for a proof that (for such a metric) $Q_{10}$ in dimension
$n=10$ coincides with the formula generated by the algorithm. In
fact, in that proof, \eqref{q8-ex} is used as a formula for $i^*
\bar{P}_2^4(1)$. The universality of \eqref{q4-ex} and \eqref{q6-ex}
has been used already in the above constructions. In the present
argument, these formulae are used in dimension $n=10$ as formulae
for the respective quantities $i^* \bar{P}_2^2(1)$ and $i^*
\bar{P}_2^3(1)$. The resulting formula for $Q_{10}$ is displayed in
Section \ref{computer}.

The argument assumes conformal flatness since some of the
factorization identities for $D_{8}^{res}(\lambda)$ and
$D_{10}^{res}(\lambda)$ which enter into the algorithm are only
known for such metrics. The problematic identities are those which
contain the factor $\bar{P}_2$ (see \eqref{fact-b} and the comments
at the end of Section \ref{theory}).

Proving universality of \eqref{q8-ex} through comparison with the
formula for $Q_8$ displayed in \cite{G-P} seems to be a challenging
task even for conformally flat metrics. Concerning a comparison of
both formula for $Q_8$ we only note that a calculation using
\eqref{q8-ex} shows that $\J^4$ contributes to $Q_8$ with the
coefficient $(\f-2)\f(\f+2)$. This observation fits with \cite{G-P}.

Next, we describe a more conceptual approach towards universality.
It rests on the systematic elaboration of the relations between the
quantities
$$
Q_{2N} \quad \mbox{and} \quad
\dot{D}^{res}_{2N}\left(-\frac{n}{2}+N\right)(1).
$$

We first describe the method by proving the universality of the
recursive formula
\begin{equation}\label{q4-univ}
Q_4 = P_2(Q_2) - 2i^* \bar{P}_2 (\bar{Q}_2)
\end{equation}
(for general metrics). For even $n \ge 8$, the polynomial
$$
Q_4^{res}(\lambda) = - D_4^{res}(\lambda)(1)
$$
can be characterized in {\em two} different ways. On the one hand,
for all even $n \ge 4$, this quadratic polynomial satisfies the
system
\begin{equation}\label{syst-1}
\begin{split}
Q_4^{res}\left(-\f\!+\!2\right) & = - P_4(1) = -\left(\f\!-\!2\right)Q_4 \\
Q_4^{res}\left(-\f\!+\!3\right) & = - P_2
D_2^{res}\left(-\f\!+\!3\right)(1)
\end{split}
\end{equation}
and the relation
\begin{equation}\label{syst-2}
Q_4^{res}\left(-\frac{n\!-\!1}{2}\right) = -
D_2^{res}\left(-\frac{n\!+\!3}{2}\right) \bar{P}_2(1).
\end{equation}
On the other hand, for even $n \ge 8$, the polynomial
$Q_4^{res}(\lambda)$ is characterized by \eqref{syst-1} and
\begin{equation}\label{syst-3}
Q_4^{res}(0) = 0.
\end{equation}
For $n=4$ and $n=6$, the condition \eqref{syst-3} is contained in
the conditions of \eqref{syst-1}. In particular, in the critical
case, these conditions do {\em not} suffice to determine the
polynomial.

For even $n \ge 4$, \eqref{syst-1} and \eqref{syst-2} imply that
\begin{equation}\label{alph}
\dot{Q}_4^{res}\left(-\f\!+\!2\right) = \frac{1}{3}
\frac{n\!-\!4}{2} Q_4 + \frac{5n\!-\!14}{6} P_2(Q_2)
-\frac{2(n\!-\!1)}{3} i^* \bar{P}_2(\bar{Q}_2).
\end{equation}
For $n=4$, this relation yields
$$
\dot{Q}_4(0) = P_2(Q_2) - 2 i^* \bar{P}_2(\bar{Q}_2).
$$
It leads to \eqref{q4-univ}, when combined with $\dot{Q}^{res}_4(0)
= Q_4$. This method has been used above. On the other hand, for even
$n \ge 8$, \eqref{syst-1} and \eqref{syst-3} imply
\begin{equation}\label{bet}
\dot{Q}_4^{res}\left(-\f\!+\!2\right) = Q_4 + \left(\f\!-\!2\right)
(Q_4 + P_2(Q_2)).
\end{equation}
Subtracting \eqref{alph} and \eqref{bet} gives
$$
0 = \frac{n\!-\!1}{3} \left(Q_4 - P_2(Q_2) - 2 i^*
\bar{P}_2(\bar{Q}_2)\right).
$$
This proves the universality of \eqref{q4-univ}. The cases $n=4,6$
are covered by analytic continuation in $n$. The argument reverses
an argument in \cite{juhl}, where \eqref{bet} was derived from
\eqref{q4-univ}.

A similar argument can be applied for $Q_6$. One formula for the
polynomial $Q_6^{res}(\lambda) = D_6^{res}(\lambda)(1)$ of degree
$3$ follows from the four factorization identities \eqref{fact-a}
and \eqref{fact-b} (for $N=3$). The calculation extends the
algorithm described above. It uses the universality of
\eqref{q4-univ}. On the other hand, for even $n \ge 12$, Lagrange's
interpolation formula yields a second formula for
$Q_6^{res}(\lambda)$ by using \eqref{fact-a} (for $N=3$) and
$$
Q_6^{res}(0)=0.
$$
For $n=6,8,10$, the latter condition is contained in the system
\eqref{fact-a} (for $N=3$).

The comparison of both resulting formulae for
$\dot{Q}_6^{res}(-\f+3)$ yields
$$
0 = \frac{n\!-\!1}{5} \left(Q_6 - \frac{2}{3} P_2(Q_4) - \frac{2}{3}
P_4(Q_2) + \frac{5}{3} P_2^2 (Q_2)  - \frac{8}{3} i^*
\bar{P}_2^2(\bar{Q}_2)\right).
$$
This proves the universality of \eqref{q6-ex}. The cases $n =6,8,10$
are covered by analytic continuation in $n$. For the details (of the
reversed argument) see \cite{juhl}, Theorems 6.11.7 -- 6.11.8.

Similarly, we compare two formulae for
$$
\dot{Q}_8^{res}\left(-\f\!+\!4\right),
$$
where $Q_8^{res}(\lambda) = - D_8^{res}(\lambda)(1)$. Under the
assumption $Q_8^{res}(0)=0$, we find
\begin{multline*}
0 = \frac{n\!-\!1}{7} \Big[Q_8 - \frac{3}{5} P_2(Q_6) + 4 P_2^2(Q_4)
- \frac{17}{5} P_4(Q_4) + \frac{22}{5} P_2^3(Q_2) \\ - \frac{8}{5}
P_2 P_4 (Q_2)  - \frac{28}{5} P_4 P_2 (Q_2) + \frac{9}{5} P_6 (Q_2)
+ \frac{16}{5} i^* \bar{P}_2^3(\bar{Q}_2)\Big].
\end{multline*}
We suppress the details of the calculations. The vanishing of the
quantity in brackets is equivalent to \eqref{q8-ex}.

The quantity $Q_8^{res}(h;0) \in C^\infty(M)$ is a scalar conformal
invariant. In fact, the conformal transformation law \eqref{covar}
implies
$$
e^{2N \varphi} D_{2N}^{res}(\hat{h};0)(1) = D_{2N}^{res}(h;0)(1), \;
\hat{h} = e^{2\varphi}h,
$$
i.e.,
\begin{equation}\label{ci}
e^{2N\varphi} Q_{2N}^{res}(\hat{h};0) = Q_{2N}^{res}(h;0)
\end{equation}
for $Q^{res}_{2N}(h;\lambda) = -(-1)^N D^{res}_{2N}(h;\lambda)$. In
particular,
\begin{equation}\label{ci-8}
e^{8\varphi} Q_8^{res}(\hat{h};0) = Q_8^{res}(h;0).
\end{equation}
By \cite{FG-final}, Section 9 there are no such non-trivial
invariants on locally conformally flat manifolds of dimension $>8$.
In other words, for locally conformally flat metrics $h$, the
condition $Q_8^{res}(h;0)=0$ is satisfied in dimension $>8$. Thus,
we have proved

\begin{propo}\label{flat-8} On locally conformally flat manifolds
$(M,h)$ of dimension $n > 8$, the recursive formula \eqref{q8-ex}
for $Q_8(h)$ holds true.
\end{propo}

An alternative method is the following. We recursively determine
$Q_8^{res}(h;\lambda)$ by factorization identities at
$$
\lambda \in \left \{ -\f\!+\!4, -\f\!+\!5, -\f\!+\!6, -\f\!+\!7
\right \} \cup \left\{ -\frac{n\!-\!1}{2} \right\}
$$
(as described by the algorithm) and evaluate the result at
$\lambda=0$. For even $n \ge 16$, the condition $Q_8^{res}(h;0)=0$
is equivalent to the universal recursive formula. Again, the cases
of even $n$ such that $8 \le n \le 14$ are covered by continuation.

As described above, Proposition \ref{flat-8} has the following
consequence.

\begin{corr}\label{rec-q10} On locally conformally flat Riemannian
manifolds of dimension $10$, the critical $Q$-curvature $Q_{10}$ is
given by the formula displayed in Section \ref{computer}.
\end{corr}

We continue with a number of supplementary comments on Conjecture
\ref{main-c}.

Alternatively, \eqref{URF} can be viewed as a formula for the
function
\begin{equation}\label{holo-term}
i^* \bar{P}_2^{N-1}(\bar{Q}_2) \in C^\infty(M)
\end{equation}
which is associated to a Poincar\'e-Einstein metric on the space
$(0,\varepsilon) \times M$. From that point of view, \eqref{URF}
states that the restriction of the function
$\bar{P}_2^{N-1}(\bar{Q}_2)$ ($N \ge 2$) to $M$ can be expressed in
terms of boundary data:
\begin{equation}\label{bulk-bound}
(-1)^N \frac{(2N\!-\!2)!!}{(2N\!-\!3)!!} i^* \bar{P}_2^{N-1}
(\bar{Q}_2) = \sum_{j=0}^{N-1} \P^{(N)}_{2j} (Q_{2N-2j}),
\end{equation}
where
\begin{equation}
\P^{(N)}_{2j} = \sum_{|I| = j} a^{(N)}_I P_{2I}.
\end{equation}
Here we use the convention that $\P^{(N)}_0 = -1$. The identity $i^*
\bar{Q}_2 = Q_2$ should be regarded as the special case $N=1$ of
these relations. The differential operators $\P_{2j}^{(N)}$ are of
the form
\begin{equation}\label{p-formula}
\alpha_j^{(N)} \Delta^j + \text{LOT}
\end{equation}
with
\begin{equation}\label{ap}
\alpha_j^{(N)} = \sum_{|I|=j} a^{(N)}_I.
\end{equation}
For the flat metric, the lower order terms in \eqref{p-formula}
vanish. In Table \ref{co}, we display the coefficients
$\alpha_j^{(N)}$ for $N \le 10$. An inspection suggests that
\begin{equation}\label{aston}
\alpha_j^{(N)} = \beta_j^{(N)},
\end{equation}
where
\begin{equation}\label{beta}
\beta_j^{(N)} \st (-1)^{j-1} {N\!-\!1 \choose j} \frac{(2j\!-\!1)!!
(2N\!-\!2j\!-\!3)!!}{(2N\!-\!3)!!}.
\end{equation}
The relations \eqref{aston} would imply the {\em symmetry relations}
\begin{equation}\label{symm}
\alpha_j^{(N)} = (-1)^{N-1} \alpha_{N-1-j}^{(N)}.
\end{equation}
These are clearly visible in Table \ref{co}. The numbers
$\beta_j^{(N)}$ have a simple generating function. Let
\begin{equation}\label{gen-func}
\G(z,w) = (1-z)^{-\frac{1}{2}} (1-w)^{-\frac{1}{2}}.
\end{equation}
Then
\begin{equation}\label{gen}
\G(z,w) = \sum_{0 \le j \le N-1} \beta_j^{(N)}
\frac{(2N\!-\!3)!!}{(2N\!-\!2)!!} (-1)^{j-1} z^j w^{N-1-j}.
\end{equation}
In fact, \eqref{beta} is equivalent to
$$
\beta_j^{(N)} = (-1)^{j-1} \frac{(2N\!-\!2)!!}{j!(N\!-\!1\!-\!j)!}
\frac{(\frac{1}{2})_j(\frac{1}{2})_{N-1-j}}{(2N\!-\!3)!!},
$$
where $(a)_n = a(a\!+\!1)\dots(a\!+\!n\!-\!1)$. But using
$$
(1\!-\!z)^{-\frac{1}{2}} = \sum_{n \ge 0} \left(\frac{1}{2}\right)_n
\frac{z^n}{n!}, \; |z| < 1,
$$
we find that the coefficient of $z^j w^{N-1-j}$ in $\G(z,w)$ is
$$
\frac{(\frac{1}{2})_j(\frac{1}{2})_{N-1-j}}{j!(N\!-\!1\!-\!j)!}.
$$
This proves \eqref{gen}. It follows that the conjectural relations
\eqref{aston} can be summarized in form of the identity
\begin{equation}\label{gen-f}
\G(z,w) = \sum_{0 \le j \le N-1}
\alpha_j^{(N)} \frac{(2N\!-\!3)!!}{(2N\!-\!2)!!} (-1)^{j-1} z^j
w^{N-1-j}
\end{equation}
of generating functions. We do not attempt to prove this identity,
but note only that it is compatible with \eqref{bulk-bound} and the
well-known fact that
\begin{equation*}\label{Q-top}
Q_{2N} = (-1)^{N-1} \Delta^{N-1}(\J),
\end{equation*}
up to terms with fewer derivatives (see \cite{comp}). Indeed, the
assertion that $\Delta^{N-1}(\J)$ contributes on both sides of
\eqref{bulk-bound} with the same weight is equivalent to the
relation
\begin{equation*}\label{alt-sum}
\sum_{j=0}^{N-1} (-1)^{j-1} \alpha_j^{(N)} = \frac{(2N\!-\!2)!!}{(2N\!-\!3)!!}.
\end{equation*}
But this identity follows from the restriction of \eqref{gen-f} to
$z=w$ by comparing the coefficients of $z^{N-1}$.

In the conformally flat case, the Taylor series of $h_r$ terminates
at the third term. More precisely,
\begin{equation}\label{terminate}
h_r = \left(1\!-\!\frac{r^2}{2} \Rho\right)^2
\end{equation}
(\cite{FG-final}, \cite{juhl}, \cite{SS}). Now \eqref{scalar}
implies
\begin{equation}\label{q-flat}
\bar{Q}_2 = \tr \left(  \left(1\!-\!\frac{r^2}{2} \Rho\right)^{-1}
\Rho\right) = \sum_{k\ge 0} \left(\frac{r^2}{2}\right)^k \tr
(\Rho^{k+1}) =  Q_2 + \frac{r^2}{2} |\Rho|^2 + \dots,
\end{equation}
and it is not hard, although it becomes tedious for large $N$, to
determine the contribution $i^* \bar{P}_2^{N-1} (\bar{Q}_2)$ to
$Q_{2N}$. We shall apply this observation in Section \ref{computer}.

We finish the present section with a brief discussion of a test of
Conjecture \ref{main-c} for {\em general} metrics. It deals with the
contributions of the powers of the Yamabe operator $\bar{P}_2$ and
extends the observation concerning the contribution of $(\B,\Rho)$
to $Q_6$ in Section \ref{intro}. Here we compare the contributions
of
\begin{equation}\label{obst}
(\Rho,\Omega^{(N-2)})
\end{equation}
to $Q_{2N}$ and
\begin{equation}\label{bar}
(-1)^{N-1} \frac{(2N\!-\!2)!!}{(2N\!-\!3)!!} i^* \bar{P}_2^{N-1}
(\bar{Q}_2).
\end{equation}
The tensor $\Omega^{(N-2)}$ is one of Graham's extended obstruction
tensors \cite{G-ext}. In particular,
$$
\Omega^{(1)} = \frac{\B}{4-n}.
$$
On the right-hand side of \eqref{URF}, the contribution \eqref{obst}
only comes from the term $i^* \bar{P}_2^{N-1}(\bar{Q}_2)$. On the
other hand, its contribution to $Q_{2N}$ can be captured by its
relation to $v_{2N}$:
\begin{equation}\label{q-v}
Q_{2N} = \cdots + (-1)^N 2^{2N-1} N! (N\!-\!1)! v_{2N}.
\end{equation}
For $2N=n$, the holographic formula \eqref{holo} is such a relation.
The suppressed lower order terms in \eqref{q-v} are not influenced
by $\Omega^{(N-2)}$. In \cite{juhl}, such extensions of \eqref{holo}
were proposed and discussed in detail for subcritical $Q_2$, $Q_4$
and $Q_6$. For $Q_8$ in dimension $n \ge 8$ we expect that
\begin{equation}\label{q8-new}
\frac{1}{2^4 4! 3!} Q_8 = 8 v_8 + 6 \T_2^*
\left(\frac{n}{2}\!-\!4\right)(v_6) + 4
\T_4^*\left(\frac{n}{2}\!-\!4\right)(v_4) +
\T_6^*\left(\frac{n}{2}\!-\!4\right)(v_2).
\end{equation}
We combine \eqref{q-v} with the fact that \eqref{obst} enters into
$v_{2N}$ with the weight
$$
\frac{(-1)^{N-1}}{2^{N-1} N!}.
$$
This follows from Graham's theory \cite{G-ext}. Hence \eqref{obst}
contributes to $Q_{2N}$ through
\begin{equation}\label{obst-q}
-2^N (N\!-\!1)! (\Rho,\Omega^{(N-2)}).
\end{equation}
Now in order to determine its contribution to \eqref{bar}, it
suffices to trace its role in
$$
i^* (\partial^2/\partial r^2)^{N-1} (\bar{Q}_2),
$$
where $\bar{Q}_2$ is given by \eqref{scalar}. Graham \cite{G-ext}
proved that the expansion
\begin{equation*}
h_r = h - \Rho r^2 + h_{(4)} r^4 + \cdots + h_{^(2N-2)} r^{2N-2} +
h_{(2N)} r^{2N} + \cdots
\end{equation*}
has the structure
\begin{equation}\label{FG-obst}
\frac{1}{2} h_{(2k)} = \frac{(-1)^k}{2^k k!} \left( \Omega^{(k-1)} +
(k\!-\!1) (\Rho \Omega^{(k-2)} + \Omega^{(k-2)} \Rho) + \cdots
\right).
\end{equation}
Thus, it suffices to consider the contributions of
$$
2 (\Rho,h_{(2N-2)}), \quad (2N\!-\!2) (\Rho,h_{(2N-2)}) \quad
\mbox{and} \quad 2N \tr (h_{(2N)})
$$
to the Taylor-coefficients of $r^{2N-1}$ in $\tr(h_r^{-1}
\dot{h}_r)$. Using \eqref{FG-obst} we find the contribution
$$
4 \frac{(-1)^{N-1}}{2^{N-1} (N\!-\!1)!} (\Rho,\Omega^{(N-2)}).
$$
It follows that
$$
i^* (\partial^2/\partial r^2)^{N-1} (\bar{Q}_2) = (-1)^N 2
\frac{(2N\!-\!2)!}{2^{N-1} (N\!-\!1)!} (\Rho,\Omega^{(N-2)}) +
\cdots,
$$
i.e., \eqref{bar} yields the contribution
$$
-2^N (N\!-\!1)! (\Rho,\Omega^{(N-2)}).
$$
It coincides with \eqref{obst-q}.

\begin{table}[p]
\begin{tabular}{c|c|c|c|c|c|c|c|c|c|c}
$j$ & $9$ & $8$ & $7$ & $6$ & $5$ & $4$ & $3$ & $2$ & $1$ & $0$  \\
\hline &&&&&&&&&& \\[-2mm]
$N=1$ & & & & & & & & & & $-1$ \\[1mm]
$N=2$ & & & & & & & & & $1$ & $-1$ \\[1mm]
$N=3$ & & & & & & & & $-1$ & $\frac{2}{3}$ & $-1$ \\[1mm]
$N=4$ & & & & & & & $1$ & $-\frac{3}{5}$ & $\frac{3}{5}$ & $-1$ \\[1mm]
$N=5$ & & & & & & $-1$ & $\frac{4}{7}$ & $-\frac{18}{35}$ &
$\frac{4}{7}$ & $-1$ \\[1mm]
$N=6$ & & & & & $1$ & $-\frac{5}{9}$ & $\frac{10}{21}$ &
$-\frac{10}{21}$ & $\frac{5}{9}$ & $-1$ \\[1mm]
$N=7$ & & & & $-1$ & $\frac{6}{11}$ & $-\frac{5}{11}$ &
$\frac{100}{231}$ & $-\frac{5}{11}$ & $\frac{6}{11}$ & $-1$ \\[1mm]
$N=8$ & & & $1$ & $-\frac{7}{13}$ & $\frac{63}{143}$ &
$-\frac{175}{429}$ & $\frac{175}{429}$ & $-\frac{63}{143}$ &
$\frac{7}{13}$ & $-1$ \\[1mm]
$N=9$ & & $-1$ & $\frac{8}{15}$ & $-\frac{28}{65}$ &
$\frac{56}{143}$ & $-\frac{490}{1287}$ & $\frac{56}{143}$ &
$-\frac{28}{65}$ & $\frac{8}{15}$ & $-1$ \\[1mm]
$N=10$ & $1$ & $-\frac{9}{17}$ & $\frac{36}{85}$ & $-\frac{84}{221}$
& $\frac{882}{2431}$ & $-\frac{882}{2431}$ & $\frac{84}{221}$ &
$-\frac{36}{85}$ & $\frac{9}{17}$ & $-1$
\end{tabular}
\bigskip

\caption{\, The coefficients $\alpha_j^{(N)}$ for $N \le
10$}\label{co}
\end{table}

\section{The structure of the coefficients $a_I^{(N)}$}\label{structure}

The right-hand sides of \eqref{URF} are generated by the algorithm
described in Section \ref{status}. In the present section, we
formulate a conjectural description of the coefficients $a_I^{(N)}$
in terms of polynomials $r_I$ which are canonically associated to
compositions $I$. These polynomials are generated by a much simpler
algorithm.

\subsection{The polynomials $r_I$ and their role}\label{gen-struc}

The polynomials $r_I$ are defined recursively as interpolation
polynomials on the sets
\begin{equation}\label{sv-k}
\SV(k) =
\left\{\frac{1}{2}-k,\dots,-\frac{1}{2},\frac{1}{2}\right\}, \; k
\in \N_0
\end{equation}
of half-integers, and on certain sets of negative integers.

First of all, we define the polynomials $r_{(k)}$, $k \in \N$. These
play the role of building blocks of the general case. Let $r_{(k)}$
be defined as the unique polynomial of degree $2k\!-\!1$ with
(simple) zeros in the integers in the interval $[-(k\!-\!1),-1]$ so
that $r_{(k)}$ is constant on $\SV(k)$, and has constant term
\begin{equation}\label{r-ct}
r_{(k)}(0) = (-1)^{k-1} \frac{(2k\!-\!3)!!}{k!}.
\end{equation}
Equivalently, $r_{(k)}$ can be defined as the interpolation
polynomial which is characterized by its $2k$ values
\begin{equation}\label{single}
\begin{split}
r_{(k)}(-i) & = 0 \qquad \mbox{for all $i = 1,\dots,k-1$}, \\
r_{(k)} \left(\frac{1}{2}-i\right) & = (-2)^{-(k-1)}
\frac{\left(\frac{1}{2}\right)_{k-1}} {(k\!-\!1)!} \qquad \mbox{for
all $i = 0,1,\dots,k$}.
\end{split}
\end{equation}
The equivalence of both characterizations follows from Lagrange's
formula.

In order to define $r_I$ for a general composition $I$, we introduce
some more notation. For any $I$, we define the rational number
\begin{equation}\label{sum-prod}
\R_I = \sum_{I=(J_1,\dots,J_M)}
r_{J_1} \left(\frac{1}{2}\right) \; \cdots \;
r_{J_M} \left(\frac{1}{2}\right),
\end{equation}
where the sum runs over all compositions $J_1, \dots, J_M$ which
form a subdivision of $I$, i.e., the sequence of natural numbers
which is obtained by writing the entries of $J_1$ followed by the
entries of $J_2$ etc., coincides with the sequence which defines
$I$. In particular, $\R_{(k)}$, $\R_{(j,k)}$ and $\R_{(i,j,k)}$ are
given by the values of the respective sums
\begin{align*}
r_{(k)}, \quad r_{(j,k)} + r_{(j)} r_{(k)} \quad \mbox{and} \quad
r_{(i,j,k)} + r_{(i,j)} r_{(k)} + r_{(i)} r_{(j,k)} + r_{(i)}
r_{(j)} r_{(k)}
\end{align*}
at $x= \frac{1}{2}$. Next, using the polynomials $r_I$, we define
\begin{multline}\label{CC}
\CC_{(I_1,\dots,I_m)}(x) \\ = r_{(I_1,\dots,I_m)}(x) + \R_{(I_1)}
\cdot r_{(I_2,\dots,I_m)}(x) + \dots + \R_{(I_1,\dots,I_{m-1})}
\cdot r_{(I_m)}(x).
\end{multline}
$\CC_I$ differs from $r_I$ by a lower degree polynomial.

Now let $r_I$ be a polynomial of degree $2|I|-1$ so that
\begin{equation}\label{third-a}
r_I(-i) = 0 \quad \mbox{for all $i = 1,\dots,|I|$, $i\ne
I_{\text{last}}$}
\end{equation}
and
\begin{equation}\label{CC-const}
\CC_I (x) \quad \mbox{is constant on $\SV(|I|)$}.
\end{equation}
Here $I_{\text{last}}$ denotes the last entry in the composition $I
= (I_{\text{first}},\dots,I_{\text{last}})$. The condition
\eqref{CC-const} constitutes the {\bf first} system of
multiplicative recursive formulae for the values of the polynomials
$r_I$.

Now \eqref{third-a} and \eqref{CC-const} determine $(|I|-1) +
(|I|+1) = 2|I|$ values of $r_I$. Since the value of $\CC_I$ on
$\SV(|I|)$ was not chosen, one additional condition is required to
characterize $r_I$. For that purpose, we use the {\bf second} system
of multiplicative recursive formulae
\begin{equation}\label{mult-2}
r_{(J,k)}(0) + r_J(k) \cdot r_{(k)}(0) = 0
\end{equation}
for the constant terms. The relations \eqref{mult-2} are required to
hold true for all $k \ge 1$ and all compositions $J$. They describe
how {\em all} values of the polynomials $r_I$ on the natural numbers
finally influence the constant terms of polynomials which are
associated to compositions of larger sizes. The values $r_{(k)}(0)$
are given by the explicit formula \eqref{r-ct}.

It follows from the above definition that $r_I$ is determined by the
(values of the) polynomials $r_J$ for all sub-compositions $J$ of
$I$. By iteration, it follows that $r_I$ is determined by the
polynomials $r_{(k)}$ for all $k$ which appear as entries of $I$.

Now we are ready to formulate the conjectural relation between the
coefficients $a_I^{(N)}$ and the values of $r_I$ on $\N$.

\begin{conj}\label{para} For all compositions $I$ and all integers $N \ge
|I|+1$,
\begin{equation}\label{ar}
a_I^{(N)} = \prod_{i=1}^{|I|}
\left(\frac{N\!-\!i}{2N\!-\!2i\!-\!1}\right) r_I(N\!-\!|I|).
\end{equation}
\end{conj}

Conjecture \ref{para} is supported by the observation that all
coefficients in the presentations \eqref{URF} of the $Q$-curvatures
$Q_{2N}$ with $N \le 14$ are correctly reproduced by \eqref{ar}. We
recall that for $Q_{28}$ the sum in \eqref{URF} already contains
$2^{13}-1$ terms.

In particular, we obtain uniform descriptions of {\em all}
coefficients in the universal recursive formulae for $Q_6$, $Q_8$
and $Q_{10}$ in terms of the polynomials $r_I$ for all compositions
$I$ with $|I| \le 4$. In Section \ref{examp}, we shall discuss these
examples in more detail.

Note that \eqref{ar} implies
$$
\alpha_j^{(N)} = \sum_{|I|=j} a_I^{(N)} = \prod_{i=1}^{j}
\left(\frac{N\!-\!i}{2N\!-\!2i\!-\!1}\right) \sum_{|I|=j}
r_I(N\!-\!|I|).
$$
Thus, under Conjecture \ref{para} the identity
\begin{align*}
\alpha_j^{(N)} & = (-1)^{j-1} {N\!-\!1 \choose j}
\frac{(2j\!-\!1)!! (2N\!-\!2j\!-\!3)!!}{(2N\!-\!3)!!} \nonumber \\
& = (-1)^{j-1} \frac{(2j\!-\!1)!!}{j!} \frac{(N\!-\!1) \dots
(N\!-\!j)} {(2N\!-\!3)\dots(2N\!-\!2j\!-\!1)}
\end{align*}
(see \eqref{aston}) is equivalent to
\begin{equation}\label{sum-r}
\sum_{|I|=j} r_I (x) = (-1)^{j-1} \frac{(2j\!-\!1)!!}{j!}.
\end{equation}

\begin{ex}\label{double} The polynomial $r_{(j,k)}$ is characterized
by its zeros in
$$
\left\{-(j+k),\dots,-1\right\} \setminus \left\{-k\right\},
$$
the constancy of
$$
\CC_{(j,k)}(x) = r_{(j,k)}(x) + r_{(j)}\left(\frac{1}{2}\right)
r_{(k)}(x)
$$
on $\SV(j+k)$, and the relation
$$
r_{(j,k)}(0) = - r_{(j)}(k) \, r_{(k)}(0).
$$
Note that $\CC_{(j,k)}$ is constant on $\SV(j+k)$ iff
\begin{equation}\label{s-double}
s_{(j,k)}(x) = - r_{(j)}\left(\frac{1}{2}\right)  s_{(k)}(x)
\end{equation}
on $\SV(j+k)$, where
\begin{equation}\label{sI}
s_I (x) = r_I (x) - r_I \left(\frac{1}{2}\right).
\end{equation}
In particular, $s_{(k,1)} = 0$ on $\SV(k+1)$.
\end{ex}

In terms of $s_I$, the condition \eqref{CC-const} is equivalent to
the condition that the polynomial
\begin{equation}\label{s-const}
s_{(I_1,\dots,I_m)}(x) + \R_{(I_1)} \cdot s_{(I_2,\dots,I_m)}(x) +
\dots + \R_{(I_1,\dots,I_{m-1})} \cdot s_{(I_m)}(x)
\end{equation}
vanishes on $\SV(|I|)$. For instance, for compositions with three
entries, \eqref{s-const} states that
\begin{equation}\label{s-triple}
s_{(i,j,k)}(x) = - r_{(i)}\left(\frac{1}{2}\right) s_{(j,k)}(x) -
\left[ r_{(i,j)} + r_{(i)} r_{(j)} \right] \left(\frac{1}{2}\right)
s_{(k)}(x)
\end{equation}
on $\SV(i+j+k)$. This generalizes \eqref{s-double}.

\eqref{s-double} implies that $s_{(j,k)}$ vanishes on $\SV(k)$ and
\begin{equation*}
s_{(j,k)}\left(-\frac{1}{2}-k\right) = - r_{(j)}
\left(\frac{1}{2}\right) s_{(k)}\left(-\frac{1}{2}-k\right).
\end{equation*}
The latter relation is a special case of
\begin{equation}\label{mult-4}
s_{(J,k)}\left(-\frac{1}{2}-k\right) = - r_J\left(\frac{1}{2}\right)
\cdot s_{(k)}\left(-\frac{1}{2}-k\right)
\end{equation}
which holds true for all compositions $J$ and all $k \ge 2$.
\eqref{mult-4} is a formula for the value of $r_{(J,k)}$ at the {\em
largest} half-integer in the set $\frac{1}{2}-\N_0$ for which this
value differs from $r_{(J,k)}(\frac{1}{2})$. It is a consequence of
\eqref{s-const}.

Finally, we note that the values of $r_I$ at $x=-I_{\text{last}}$
satisfy the {\bf third} system of multiplicative recursive relations
\begin{equation}\label{mult-1}
r_{(J,k,j)}(-j) = - r_J(k) \cdot r_{(k,j)}(-j)
\end{equation}
for all $j, k \ge 1$ and all compositions $J$. We summarize both
relations \eqref{mult-2} and \eqref{mult-1} in
\begin{equation}\label{mult}
r_{(J,k,j)}(-j) = - r_J(k) \cdot r_{(k,j)}(-j)
\end{equation}
for all $j\ge 0$, $k \ge 1$ and all $J$. Here we use the convention
$r_{(I,0)} = r_I$. With the additional convention $r_{(0)}=-1$,
\eqref{mult} makes sense also for $J=(0)$.

\subsection{Examples}\label{examp}

In the present section, we explicate and confirm Conjecture
\ref{para} in a number of important special cases.

\subsubsection{The polynomials $r_I$ for $|I| \le 4$}\label{pol-4}

We determine the polynomials $r_I$ which are responsible for the
coefficients in the universal recursive formulae for $Q_{2N}$, $N
\le 5$. These are the polynomials $r_I$ for all compositions $I$ of
size $|I| \le 4$.

\begin{ex}\label{ex1} We consider the polynomials $r_I$ for
compositions $I$ of size $|I| \le 2$. First of all, $r_{(1)}=1$. The
polynomials $r_{(1,1)}$ and $r_{(2)}$ for compositions $I$ of size
$|I|=2$ are listed in Table \ref{r2}. They are characterized as
follows by their properties. Both polynomials are of degree $2|I|-1
= 3$ and satisfy the respective relations
\begin{align*}
r_{(1,1)}\left(\frac{1}{2}\right) & =
r_{(1,1)}\left(-\frac{1}{2}\right) =
r_{(1,1)}\left(-\frac{3}{2}\right) = - \frac{5}{4}, \\
r_{(1,1)}(-2) & = 0,
\end{align*}
and
\begin{align*}
r_{(2)} \left(\frac{1}{2}\right) & =
r_{(2)}\left(-\frac{1}{2}\right) =
r_{(2)}\left(-\frac{3}{2}\right) = - \frac{1}{4}, \\
r_{(2)}(-1) & = 0
\end{align*}
(see Table \ref{v2}). The values $-\frac{5}{4}$ and $-\frac{1}{4}$
are given by
$$
-\frac{5}{4} = (-2)^{-1} \left(\frac{1}{2}+2\right)_1 \quad
\mbox{and} \quad -\frac{1}{4} = (-2)^{-1}
\left(\frac{1}{2}\right)_1,
$$
respectively (see \eqref{sigma-value}). Alternatively, the value of
$r_{(1,1)}$ on the set $\SV(2)$ is determined by the recursive
relation
$$
r_{(1,1)}(0) = - r_{(1)}(1) \cdot r_{(1)}(0) = - 1
$$
(see \eqref{mult-2}) for its constant term. Similarly, the value of
$r_{(2)}$ on the set $\SV(2)$ can be determined by the relation
$r_{(2)}(0) = -\frac{1}{2}$ (see \eqref{r-ct}).
\end{ex}

\begin{ex}\label{ex2} The polynomials
$$
r_{(1,1,1)}, \; r_{(1,2)}, \; r_{(2,1)}, \; r_{(3)}
$$
for compositions $I$ of size $|I|=3$ are listed in Table \ref{r3}.
These four polynomials of degree $5$ are determined as follows by
their properties. First of all, $r_{(3)}$ and $r_{(2,1)}$ are
characterized by their respective zeros in $x=-1,-2$ and $x=-2,-3$,
and their respective values
$$
\frac{3}{32} = (-2)^{-2} \frac{\left(\frac{1}{2}\right)_2}{2!} \quad
\mbox{and} \quad \frac{7}{16} = (-2)^{-2}
\frac{\left(\frac{1}{2}\right)_1 \left(\frac{1}{2}+3\right)_1}{1!1!}
$$
on the set $\SV(3)$ (see \eqref{sigma-value} and Table \ref{v3}).
Alternatively, $r_{(3)}$ is constant on $\SV(3)$, and the value of
the constant is determined by its constant term $r_{(3)}(0) =
\frac{1}{2}$ (see \eqref{r-ct}). The values of $r_{(2,1)}$ on
$\SV(3)$ are determined by the constancy of
$$
\CC_{(2,1)} = r_{(2,1)} + r_{(2)} \left(\frac{1}{2}\right) r_{(1)} =
r_{(2,1)} - \frac{1}{4}
$$
on this set, and the relation
$$
r_{(2,1)}(0) = - r_{(2)}(1) \cdot r_{(1)}(0) = -1
$$
(see \eqref{mult-2}). Similarly, $r_{(1,2)}$ is characterized by its
zeros in $x=-1,-3$, the constancy of
$$
\CC_{(1,2)} = r_{(1,2)} + r_{(1)} \left(\frac{1}{2}\right) r_{(2)} =
r_{(1,2)} + r_{(2)},
$$
on the set $\SV(3)$, and the relation
$$
r_{(1,2)}(0) = - r_{(1)}(2) \cdot r_{(2)}(0) = - r_{(2)}(0)
$$
(see \eqref{mult-2}). These are special cases of Example
\ref{double}. Finally, $r_{(1,1,1)}$ has zeros in $x=-2,-3$,
$\CC_{(1,1,1)}$ is constant on $\SV(3)$, i.e., $r_{(1,1,1)} +
r_{(1,1)}$ is constant on $\SV(3)$, and
$$
r_{(1,1,1)}(0) = - r_{(1,1)}(1) \cdot r_{(1)}(0) = - r_{(1,1)}(1)
$$
(see \eqref{mult-2}).
\end{ex}

\begin{ex}\label{ex3} The polynomials $r_I$ for compositions $I$
of size $|I|=4$ are listed in Table \ref{r4}. We characterize these
eight degree $7$ polynomials in terms of their properties. Their
values on $\SV(4)$ are displayed in Table \ref{v4}. First of all,
the interpolation polynomial $r_{(4)}$ is defined as in
\eqref{single}. A special case of Example \ref{double} yields a
characterization of $r_{(3,1)}$. In particular, $s_{(3,1)} = 0$ on
$\SV(4)$. Note also that $r_{(4)}$ and $r_{(3,1)}$ coincide with the
averages $\sigma_{(4,4)}$ and $\sigma_{(3,4)}$ (see
\eqref{special-sigma}). These polynomials can be characterized as in
Section \ref{aver} by their zeros and their values on the set
$\SV(4)$. The polynomials $r_{(2,2)}$ and $r_{(1,3)}$ are also
covered by Example \ref{double}. The central facts are that
$\CC_{(2,2)}$ and $\CC_{(1,3)}$ are constant on $\SV(4)$. We recall
that this is equivalent to
$$
s_{(2,2)} = - \R_{(2)} \cdot s_{(2)} \quad \mbox{and} \quad
s_{(1,3)} = - \R_{(1)} \cdot s_{(3)}
$$
on $\SV(4)$. Next, the polynomials $r_{(2,1,1)}$ and $r_{(1,2,1)}$
both have zeros in $\left\{-2,-3,-4\right\}$. Moreover, the
functions
\begin{equation*}
\CC_{(2,1,1)} = r_{(2,1,1)} + \R_{(2)} \cdot r_{(1,1)} + \R_{(2,1)}
\cdot  r_{(1)}
\end{equation*}
and
\begin{equation*}
\CC_{(1,2,1)} = r_{(1,2,1)} + \R_{(1)} \cdot r_{(2,1)} + \R_{(1,2)}
\cdot r_{(1)}
\end{equation*}
are constant on the set $\SV(4)$ (see \eqref{CC}), and we have the
recursive relations
$$
r_{(2,1,1)}(0) = - r_{(2,1)}(1) \cdot r_{(1)}(0) \quad \mbox{and}
\quad r_{(1,2,1)}(0) = - r_{(1,2)}(1) \cdot r_{(1)}(0)
$$
for the constant terms (see \eqref{mult-2}). Note that
$\CC_{(2,1,1)}$ and $\CC_{(1,2,1)}$ are constant on $\SV(4)$ iff
$$
s_{(2,1,1)} = - \R_{(2)} \cdot s_{(1,1)} \quad \mbox{and} \quad
s_{(1,2,1)} = - \R_{(1)} \cdot s_{(2,1)}
$$
on $\SV(4)$, respectively (see \eqref{s-const}). Similar arguments
apply to $r_{(1,1,1,1)}$ and $r_{(1,1,2)}$. These polynomials vanish
on the respective sets
$$
\left\{-2,-3,-4\right\} \quad \mbox{and} \quad
\left\{-1,-3,-4\right\},
$$
the functions
$$
\CC_{(1,1,2)} = r_{(1,1,2)} + \R_{(1)} \cdot r_{(1,2)} + \R_{(1,2)}
\cdot r_{(2)}
$$
and
$$
\CC_{(1,1,1,1)} = r_{(1,1,1,1)} + \R_{(1)} \cdot r_{(1,1,1)} +
\R_{(1,1)} \cdot r_{(1,1)} + \R_{(1,1,1)} \cdot r_{(1)}
$$
are constant on $\SV(4)$, and
$$
r_{(1,1,1,1)}(0) = - r_{(1,1,1)}(1) \cdot r_{(1)}(0) \quad
\mbox{and} \quad  r_{(1,1,2)}(0) = - r_{(1,1)}(2) \cdot r_{(2)}(0)
$$
(see \eqref{mult-2}). Note that $\CC_{(1,1,2)}$ and
$\CC_{(1,1,1,1)}$ are constant on $\SV(4)$ iff
$$
s_{(1,1,2)} = - \R_{(1)} \cdot s_{(1,2)} \quad \mbox{and} \quad
s_{(1,1,1,1)} = - \R_{(1)} \cdot s_{(1,1,1)} - \R_{(1,1)} \cdot
s_{(1,1)},
$$
respectively (see \eqref{s-const}). The listed properties of $s_I$
and $r_I$ can be easily verified using Tables \ref{v2} -- \ref{v4}
and Tables \ref{vi2} -- \ref{vi4}. Here we use $\R_{(1,1)} = -
\frac{1}{4}$ and $\R_{(1,1,1)} = \frac{1}{32}$.
\end{ex}

These results can be used to confirm Conjecture \ref{para} for the
coefficients in the universal formulae for $Q_{2N}$ for $N \le 5$.
For the calculations of the values of the polynomials $r_I$ we apply
the formulae in Table \ref{r2} -- Table \ref{r4}.

\begin{ex}\label{ex-q6} By \eqref{ar}, the three coefficients in the
formula \eqref{q6-ex} for $Q_6$ are given by
\begin{align*}
a_{(1)}^{(3)} & = \frac{2}{3} \cdot r_{(1)}(2) = \frac{2}{3}, \\
a_{(1,1)}^{(3)} & = \frac{2}{3} \cdot r_{(1,1)}(1) = \frac{2}{3}
\cdot \left( -\frac{5}{2} \right) = -\frac{5}{3}, \\
a_{(2)}^{(3)} & = \frac{2}{3} \cdot r_{(2)}(1) = \frac{2}{3}.
\end{align*}
\end{ex}

\begin{ex}\label{q-8} By \eqref{ar}, the seven coefficients in the formula
\eqref{q8-ex} for $Q_8$ are given by the following formulae. First
of all,
$$
a_{(1)}^{(4)} = \frac{3}{5} \cdot r_{(1)}(3) = \frac{3}{5}.
$$
Next,
\begin{align*}
a_{(1,1)}^{(4)} & = \frac{3 \cdot 2}{5 \cdot 3} \cdot r_{(1,1)}(2) =
-4, \\
a_{(2)}^{(4)} & = \frac{3 \cdot 2}{5 \cdot 3} \cdot r_{(2)}(2) =
\frac{2}{5} \cdot \frac{17}{2} = \frac{17}{5}
\end{align*}
and
\begin{align*}
a_{(2,1)}^{(4)} & = \frac{3 \cdot 2}{5 \cdot 3} \cdot r_{(2,1)}(1) =
\frac{2}{5} \cdot 14 = \frac{28}{5}, \\
a_{(3)}^{(4)} & = \frac{3 \cdot 2}{5 \cdot 3} \cdot r_{(3)}(1) = -
\frac{2}{5} \cdot \frac{9}{2} = - \frac{9}{5}.
\end{align*}
Finally,
$$
a_{(1,1,1)}^{(4)} = \frac{3 \cdot 2}{5 \cdot 3} \cdot r_{(1,1,1)}(1)
= -\frac{22}{5} \quad \mbox{and} \quad  a_{(1,2)}^{(4)} = \frac{3
\cdot 2}{5 \cdot 3} \cdot r_{(1,2)}(1) = \frac{8}{5}.
$$
\end{ex}

\begin{ex}\label{q-10} The fifteen coefficients in the universal recursive
formula for $Q_{10}$ (see Section \ref{computer}) are determined by
the values of the polynomials $r_I$ with $|I| \le 4$ at certain
integers. In particular,
\begin{align*}
a_{(1,3)}^{(5)} & = \frac{4!}{105} \cdot r_{(1,3)}(1) = - \frac{69}{35}, \\
a_{(2,1)}^{(5)} & = \frac{4!}{105} \cdot r_{(2,1)}(2) =
\frac{176}{5}
\end{align*}
and
$$
a_{(2)}^{(5)} = \frac{12}{35} \cdot r_{(2)}(3) = \frac{312}{35}.
$$
Similar straightforward calculations reproduce the remaining twelve
coefficients.
\end{ex}

\subsubsection{Some closed formulae}\label{closed}

For some compositions, Conjecture \ref{para} allows to derive closed
formulae for the coefficients in the universal recursive formulae.
Here we discuss such formulae for the coefficients of the extreme
contributions $P_2(Q_{2N-2})$ and $P_{2N-2}(Q_2)$.

\begin{lemm}\label{expl-case} Under Conjecture \ref{para},
$$
a_{(1)}^{(N)} = \alpha_{(1)}^{(N)} = \frac{N\!-\!1}{2N\!-\!3} \quad
\mbox{and} \quad  a_{(N-1)}^{(N)} = (-1)^{N-1}
\frac{N\!-\!1}{2N\!-\!3} (2N\!-\!5)
$$
for $N \ge 2$.
\end{lemm}

\begin{proof} The first formula follows from $r_{(1)} = 1$. The second
claim is a consequence of the Lagrange representation of
$r_{(N-1)}$. By \eqref{ar},
\begin{equation}\label{a-rep}
a_{(N-1)}^{(N)} = \prod_{i=1}^{N-1}
\left(\frac{N\!-\!i}{2N\!-\!2i\!-\!1} \right) r_{(N-1)}(1),
\end{equation}
where the polynomial $r_{(N-1)}$ is characterized by \eqref{single}.
Now by Lagrange's formula,
$$
r_{(k)}(x) = (-2)^{-(k-1)} \frac{\left(\frac{1}{2}\right)_{k-1}}
{(k\!-\!1)!} \sum_{i=0}^k \prod_{j=0, \; j \ne i}^k \left(
\frac{x+j-\frac{1}{2}}{j-i}\right) \prod_{j=1}^{k-1}
\left(\frac{x+j}{j+\frac{1}{2}-i}\right).
$$
Hence
$$
r_{(k)}(1) = (-2)^{-(k-1)} \frac{\left(\frac{1}{2}\right)_{k-1}}
{(k\!-\!1)!} \sum_{i=0}^k \prod_{j=0, \; j \ne i}^k \left(
\frac{j+\frac{1}{2}}{j-i}\right) \prod_{j=1}^{k-1}
\left(\frac{j+1}{j+\frac{1}{2}-i}\right).
$$
A calculation shows that the latter formula is equivalent to
$$
r_{(k)}(1) = (-1)^k 2^{-(2k-1)} \frac{(2k\!-\!3)!!}{(k\!-\!1)!}
\sum_{i=0}^k (2i\!-\!1) \begin{pmatrix} 2k+1 \\2i+1 \end{pmatrix}.
$$
It follows that
$$
r_{(N-1)}(1) = (-1)^{N-1} 2^{-(2N-3)}
\frac{(2N\!-\!5)!!}{(N\!-\!2)!} \sum_{i=0}^{N-1} (2i\!-\!1)
\begin{pmatrix} 2N-1 \\2i+1 \end{pmatrix}.
$$
Hence by \eqref{a-rep},
$$
a_{(N-1)}^{(N)} = (-1)^{N-1} 2^{-(2N-3)} \frac{N\!-\!1}{2N\!-\!3}
\sum_{i=0}^{N-1} (2i\!-\!1)
\begin{pmatrix} 2N-1 \\ 2i+1 \end{pmatrix},
$$
i.e., the assertion is equivalent to
$$
\sum_{i=0}^{N-1}(2i\!-\!1) \begin{pmatrix} 2N-1 \\2i+1
\end{pmatrix} = (2N\!-\!5) 2^{2N-3}.
$$
The latter identity follows by subtracting
$$
2 \sum_{i=0}^{N-1} \begin{pmatrix} 2N-1 \\2i+1 \end{pmatrix} =
2^{2N-1}
$$
from half of the difference of
$$
\sum_{i=0}^{2N-1} i \begin{pmatrix} 2N\!-\!1 \\ i \end{pmatrix}  =
(2N\!-\!1) 2^{2N-2} \;\; \mbox{and} \;\; \sum_{i=0}^{2N-1} (-1)^i i
\begin{pmatrix} 2N\!-\!1 \\ i \end{pmatrix}  = 0.
$$
The proof is complete. \end{proof}

\subsubsection{On the multiplicative relations for the constant
terms} \label{sec-mult}

The first system of multiplicative recursive relations concerns the
values of the polynomials $r_I$ on the set $\SV(|I|)$ of
half-integers. Their role was already exemplified in Examples
\ref{ex1} -- \ref{ex3}. The second and the third system of
multiplicative recursive relations concern the values of the
polynomials $r_I$ at $x=0$ and $x=-I_{\text{last}}$. The constant
terms satisfy the relations
\begin{equation}\label{mult-first}
r_{(J,k)}(0) = -r_J(k) \cdot r_{(k)}(0).
\end{equation}

\begin{ex}\label{ex-mult1} We use \eqref{mult-first} to determine the
constant values $r_I(0)$ of the polynomials $r_I$ for all
compositions $I$ with $|I|=5$. These values are listed in Table
\ref{vi5}. From this table it is evident that the values
$-r_{(J,1)}(0)$ with $|J|=4$ coincide with the values which are
listed in Table \ref{vi4} for $x=1$. Similarly, the values
$r_{(J,2)}(0)$ with $|J|=3$ easily follow from the values in Table
\ref{vi3} for $x=2$ using $r_{(2)}(0)=-\frac{1}{2}$ and
\eqref{mult-first}. Finally, the values $r_{(J,3)}(0)$ with $|J|=2$
follow from the values in Table \ref{vi2} for $x=3$ using
$r_{(3)}(0) = \frac{1}{2}$ and \eqref{mult-first}.
\end{ex}

\section{Further comments}\label{open}

The treatment of $Q$-curvatures in the present paper suggests a
number of further studies. Some of these are summarized in the
following.

Of course, the main open problems are Conjecture \ref{main-c} and
Conjecture \ref{para}.

The proposed universal recursive formulae for $Q$-curvatures involve
respective lower order $Q$-curvatures and lower order
GJMS-operators. These formulae can be made more explicit by
combining them with formulae for GJMS-operators. For the discussion
of recursive formulae for these operators (as well as alternative
recursive formulae for $Q$-curvatures) we refer to
\cite{juhl-power}.

All recursive formulae for $Q$-curvatures involve a term which is
defined through a power of the Yamabe operator $\bar{P}_2$. Its
structure remains to be studied.

In Section \ref{status}, the universality of the recursive formulae
for $Q_4$, $Q_6$ and $Q_8$ was proved (for locally conformally flat
metrics) by comparing {\em two} formulae for the respective
quantities $\dot{Q}_{2N}^{res}(-\f+N)$, $N=2,3,4$. This method
deserves a further development. In fact, it should yield a full
proof of the universality. Along this way, computer assisted
calculations confirm the universality (in the locally conformally
flat category) for not too large $N$.

Through Conjecture \ref{para}, the coefficients in the recursive
formulae for $Q$-curvatures are linked to interpolation polynomials
$r_I$ which are characterized by their values on integers and
half-integers in $[-|I|,1]$. A {\em conceptual} explanation of that
description is missing.

The polynomials $r_I$ should be explored systematically. In
particular, the identity \eqref{sum-r} and the properties of the
averages $\sigma_{(k,j)}$ formulated in Section \ref{aver} remain to
be proved.

The coefficients $\alpha_j^{(N)}$ are expected to have a nice
generating function $\G$ (see \eqref{gen-f}). Can one phrase the
structure of the polynomials $r_I$ in terms of generating functions,
too? In particular, it seems to be natural to study the generating
function
\begin{align*}
Q({\bf{x}};y) & = \sum_{0 \le |I| \le N-1}
\frac{(2N\!-\!3)!!}{(2N\!-\!2)!!} \, a_I^{(N)} {\bf{x}}^I
y^{N-1-|I|}, \; {\bf{x}} = (x_1,x_2,\dots).
\end{align*}
This function refines $\G$. In fact, for ${\bf{x}} = \diag (x) =
(x,x,\dots)$, \eqref{ap} and \eqref{gen-f} imply
$$
Q(\diag(x);y) = - \G(-x,y).
$$
Under Conjecture \ref{para}, $Q(\cdot;\cdot)$ can be expressed in
terms of the polynomials $r_I$. A calculation shows that
$$
Q({\bf{x}};y) = \sum_{I} \frac{1}{2^{|I|}} \left( \sum_{N \ge 0}
\left(\frac{1}{2}\right)_N r_I(N\!+\!1) \frac{y^N}{N!} \right)
{\bf{x}}^I.
$$

\section{Appendix}\label{app}

In the present section, we describe part of the numerical data which
led to the formulation of Conjecture \ref{main-c} and Conjecture
\ref{para}. We start with explicit versions of the universal
recursive formulae for $Q_{2N}$ with $N = 5,\dots,8$. Then we
describe a test of the universality of the recursive formulae for
round spheres. We display the polynomials $r_I$ and their values on
integers and half-integers for compositions $I$ with $|I| \le 5$.
Finally, we formulate some remarkable properties of the averages of
the polynomials $r_I$ over certain sets of compositions.

\subsection{Explicit formulae for $Q_{2N}$ for $N \le 8$}\label{computer}

Explicit versions of the universal recursive formulae for $Q_{2N}$
for $N=2,3,4$ were given in Section \ref{intro}. Here we add the
corresponding universal recursive formulae for $Q_{10}$, $Q_{12}$,
$Q_{14}$ and $Q_{16}$. These formulae are generated by the algorithm
of Section \ref{status}, i.e., the displayed formulae for higher
order $Q$-curvatures $Q_{2N}$ are to be understood in the sense of
Conjecture \ref{main-c} stating that the generated expressions
coincide with $Q$-curvature.

In dimension $n=10$, the algorithm yields the following formula for
$Q_{10}$ with $16$ terms.
\medskip

\begin{sloppypar} \noindent\(\frac{4}{7} P_2 (Q_{8}) - \frac{66}{7} P_2^2
(Q_6) - \frac{184}{5} P_2^3 (Q_4) - \frac{2012}{35} P_2^4 (Q_2) +
\frac{312}{35} P_4 (Q_6) - \frac{908}{35} P_4^2 (Q_2) -
\frac{456}{35} P_6 (Q_4) + \frac{20}{7} P_8 (Q_2) + \frac{76}{5} P_2
P_4 (Q_4) - \frac{69}{35} P_2 P_6 (Q_2) + \frac{176}{7} P_2^2 P_4
(Q_2) + \frac{176}{5} P_4 P_2 (Q_4 ) + \frac{376}{7} P_4 P_2^2 (Q_2)
- \frac{594}{35} P_6 P_2 (Q_2) + \frac{688}{35} P_2 P_4 P_2 (Q_2) +
\frac{128}{35} i^* \bar{P}_2^4 (\bar{Q}_2). \)
\end{sloppypar}
\medskip

The derivation of this formula assumes, in particular, that
\eqref{q8-ex} for $Q_8$ holds true in dimension $n=10$. By
Proposition \ref{flat-8}, this assumption is satisfied for
conformally flat metrics. Hence the above formula is proved for such
metrics (Corollary \ref{rec-q10}). Conjecture \ref{main-c} states
that the formula is universally true for $n \ge 10$.

Next, the algorithm yields the following conjectural formula for
$Q_{12}$ in dimension $n=12$ with $32$ terms.
\medskip

\begin{sloppypar} \noindent\(\frac{5}{9} P_2 (Q_{10}) - \frac{1180}{63}
P_2^2 (Q_8) - \frac{442}{3} P_2^3 (Q_6) - \frac{38312}{63} P_2^4
(Q_4) - \frac{8260}{9} P_2^5 (Q_2) + \frac{1150}{63} P_4 (Q_8) -
\frac{18533}{63} P_4^2 (Q_4) - \frac{356}{7} P_6 (Q_6) +
\frac{1990}{63} P_8 (Q_4) - \frac{35}{9} P_{10} (Q_2) +
\frac{208}{3} P_2 P_4 (Q_6) - \frac{1576}{9} P_2 P_4^2(Q_2) -
\frac{276}{7} P_2 P_6 (Q_4) + \frac{152}{63} P_2 P_8 (Q_2) +
\frac{18980}{63} P_2^2 P_4 (Q_4) - \frac{1555}{21} P_2^2 P_6 (Q_2) +
\frac{2832}{7} P_2^3 P_4 (Q_2) + \frac{388}{3} P_4 P_2 (Q_6) +
\frac{33680}{63} P_4 P_2^2 (Q_4) + \frac{50968}{63} P_4 P_2^3 (Q_2)
+ \frac{524}{7} P_4 P_6 (Q_2) - \frac{3556}{9} P_4^2 P_2 (Q_2) -
\frac{1116}{7} P_6 P_2 (Q_4) - \frac{1690}{7} P_6 P_2^2 (Q_2) +
\frac{2672}{21} P_6 P_4 (Q_2) + \frac{2420}{63} P_8 P_2 (Q_2) +
\frac{1632}{7} P_2 P_4 P_2 (Q_4) + \frac{22160}{63} P_2 P_4 P_2^2
(Q_2) - \frac{1027}{21} P_2 P_6 P_2 (Q_2) + \frac{25520}{63} P_2^2
P_4 P_2 (Q_2) - \frac{22432}{63} P_4 P_2 P_4 (Q_2) - \frac{256}{63}
i^* \bar{P}_2^5 (\bar{Q}_2). \)
\end{sloppypar}
\medskip

This formula for $Q_{12}$ was derived under the assumptions that the
above formulae for $Q_8$ and $Q_{10}$ hold true in dimension $n=12$.
Computer calculations confirm this assumption for locally
conformally flat metrics (see the comment in Section \ref{open}).
Conjecture \ref{main-c} states that the above formula is universally
true for $n \ge 12$.

The following formulae for $Q_{14}$ and $Q_{16}$ contain $64$ and
$128$ terms, respectively. Their generation assumes that the above
formulae for $Q_8, \dots, Q_{12}$ hold true in the respective
dimensions $14$ and $16$. Conjecture \ref{main-c} asserts that these
formulae are universal.

$Q_{14}$ is given by
\medskip

\begin{sloppypar} \noindent\(\frac{6}{11} P_2 (Q_{12}) -\frac{1085}{33} P_2^2 (Q_{10})
-\frac{14140}{33} P_2^3 (Q_8) -\frac{256362}{77} P_2^4 (Q_6)
-\frac{444680}{33} P_2^5 (Q_4) -\frac{4685236}{231} P_2^6 (Q_2)
+\frac{1070}{33} P_4 (Q_{10}) -\frac{127068}{77} P_4^2 (Q_6)
+\frac{965266}{231} P_4^3 (Q_2) -\frac{11260}{77} P_6 (Q_8)
-\frac{41058}{77} P_6^2 (Q_2) +\frac{13540}{77} P_8 (Q_6)
-\frac{2050}{33} P_{10} (Q_4) +\frac{54}{11} P_{12} (Q_2)
+\frac{7180}{33} P_2 P_4 (Q_8) -\frac{99842}{33} P_2 P_4^2 (Q_4)
-\frac{21594}{77} P_2 P_6 (Q_6) +\frac{1700}{21} P_2 P_8 (Q_4)
-\frac{95}{33} P_2 P_{10} (Q_2) +\frac{135600}{77} P_2^2 P_4 (Q_6)
-\frac{93560}{21} P_2^2 P_4^2 (Q_2) -\frac{100900}{77} P_2^2 P_6
(Q_4) +\frac{39016}{231} P_2^2 P_8 (Q_2) +\frac{1547996}{231} P_2^3
P_4 (Q_4) -1657 P_2^3 P_6 (Q_2) +\frac{691568}{77} P_2^4 P_4 (Q_2)
+\frac{11800}{33} P_4 P_2 (Q_8) +\frac{214980}{77} P_4 P_2^2 (Q_6)
+\frac{2615216}{231} P_4 P_2^3 (Q_4) +\frac{562552}{33} P_4 P_2^4
(Q_2) +\frac{99632}{77} P_4 P_6 (Q_4) -\frac{39220}{231} P_4 P_8
(Q_2) -\frac{61304}{11} P_4^2 P_2 (Q_4) -\frac{1938340}{231} P_4^2
P_2^2 (Q_2) -\frac{62019}{77} P_6 P_2 (Q_6) -\frac{251820}{77} P_6
P_2^2 (Q_4) -\frac{379314}{77} P_6 P_2^3 (Q_2) +\frac{150004}{77}
P_6 P_4 (Q_4) +\frac{10520}{21} P_8 P_2 (Q_4) +\frac{174380}{231}
P_8 P_2^2 (Q_2) -\frac{96380}{231} P_8 P_4 (Q_2) -\frac{2405}{33}
P_{10} P_2 (Q_2) +\frac{102888}{77} P_2 P_4 P_2 (Q_6)
+\frac{1247200}{231} P_2 P_4 P_2^2 (Q_4) +\frac{1876304}{231} P_2
P_4 P_2^3 (Q_2) +\frac{8808}{11} P_2 P_4 P_6 (Q_2)
-\frac{133528}{33} P_2 P_4^2 P_2 (Q_2) -\frac{5914}{7} P_2 P_6 P_2
(Q_4) -\frac{97765}{77} P_2 P_6 P_2^2 (Q_2) +\frac{52056}{77} P_2
P_6 P_4 (Q_2) +\frac{22376}{231} P_2 P_8 P_2 (Q_2)
+\frac{1372640}{231} P_2^2 P_4 P_2 (Q_4) +\frac{2066000}{231} P_2^2
P_4 P_2^2 (Q_2) -\frac{126565}{77} P_2^2 P_6 P_2 (Q_2)
+\frac{690288}{77} P_2^3 P_4 P_2 (Q_2) -\frac{1299544}{231} P_4 P_2
P_4 (Q_4) +\frac{15278}{11} P_4 P_2 P_6 (Q_2) -\frac{580960}{77} P_4
P_2^2 P_4 (Q_2) +\frac{124956}{77} P_4 P_6 P_2 (Q_2)
+\frac{15256}{7} P_6 P_2 P_4 (Q_2) + 2608 P_6 P_4 P_2 (Q_2)
-\frac{831296}{231} P_2 P_4 P_2 P_4 (Q_2) -\frac{1739296}{231} P_4
P_2 P_4 P_2 (Q_2) + \frac{1024}{231} i^*\bar{P}_2^6 (\bar{Q}_2) \).
\end{sloppypar}
\medskip

$Q_{16}$ is given by
\medskip

\begin{sloppypar} \noindent \( \frac{7}{13} P_2 (Q_{14}) -\frac{7560}{143} P_2^2
(Q_{12}) -\frac{440020}{429} P_2^3(Q_{10})
-\frac{1831120}{143}P_2^4(Q_8) -\frac{13946520}{143} P_2^5(Q_6)
-\frac{168379936}{429}P_2^6(Q_4) -\frac{253032464}{429} P_2^7(Q_2)
+\frac{7497}{143}P_4(Q_{12}) -\frac{917380}{143} P_4^2(Q_8)
+\frac{37930786}{429}P_4^3(Q_4) -\frac{49735}{143} P_6(Q_{10})
-\frac{1688928}{143}P_6^2(Q_4) +\frac{292925}{429} P_8(Q_8)
-\frac{67235}{143}P_{10}(Q_6) +\frac{15393}{143} P_{12}(Q_4)
-\frac{77}{13}P_{14}(Q_2) +\frac{234640}{429}P_2P_4(Q_{10})
-\frac{3427872}{143} P_2P_4^2(Q_6)
+\frac{25852064}{429}P_2P_4^3(Q_2) -\frac{178440}{143} P_2P_6(Q_8)
-\frac{630732}{143}P_2P_6^2(Q_2) +\frac{10720}{13} P_2P_8(Q_6)
-\frac{4760}{33}P_2P_{10}(Q_4) +\frac{480}{143} P_2P_{12}(Q_2)
+\frac{1010000}{143}P_2^2P_4(Q_8) -\frac{41694760}{429}
P_2^2P_4^2(Q_4) -\frac{137640}{13}P_2^2P_6(Q_6) +\frac{1778320}{429}
P_2^2P_8(Q_4) -\frac{142100}{429}P_2^2P_{10}(Q_2)
+\frac{7409088}{143} P_2^3P_4(Q_6)
-\frac{55899776}{429}P_2^3P_4^2(Q_2) -\frac{5536944}{143}
P_2^3P_6(Q_4) +\frac{2179520}{429}P_2^3P_8(Q_2) +\frac{7634864}{39}
P_2^4P_4(Q_4) -\frac{6964156}{143}P_2^4P_6(Q_2)
+\frac{112242560}{429} P_2^5P_4(Q_2)
+\frac{354760}{429}P_4P_2(Q_{10}) +\frac{1484320}{143} P_4P_2^2(Q_8)
+\frac{11325168}{143}P_4P_2^3(Q_6) +\frac{136807744}{429}
P_4P_2^4(Q_4) +\frac{205619360}{429}P_4P_2^5(Q_2) +\frac{132576}{13}
P_4P_6(Q_6) -\frac{1765000}{429} P_4 P_8 (Q_4) +\frac{142520}{429}
P_4P_{10}(Q_2) -39144 P_4^2 P_2 (Q_6) -\frac{67612960}{429}
P_4^2P_2^2(Q_4) -\frac{101618576}{429} P_4^2 P_2^3 (Q_2)
-\frac{3353544}{143} P_4^2 P_6 (Q_2) +\frac{50673224}{429} P_4^3 P_2
(Q_2) -\frac{418680}{143} P_6 P_2 (Q_8) -\frac{3198660}{143} P_6
P_2^2 (Q_6) -\frac{12885264}{143} P_6 P_2^3 (Q_4)
-\frac{19368456}{143} P_6 P_2^4 (Q_2) +\frac{2036928}{143} P_6 P_4
(Q_6) -\frac{5127424}{143} P_6 P_4^2 (Q_2) +\frac{236480}{143} P_6
P_8 (Q_2) -\frac{192312}{13} P_6^2 P_2 (Q_2) +\frac{43480}{13} P_8
P_2 (Q_6) +\frac{5782880}{429}P_8 P_2^2 (Q_4) +\frac{8693680}{429}
P_8 P_2^3 (Q_2) -\frac{3602780}{429}P_8 P_4 (Q_4)
+\frac{340920}{143} P_8 P_6 (Q_2) -\frac{41720}{33} P_{10} P_2 (Q_4)
-\frac{815500}{429} P_{10}P_2^2 (Q_2) +\frac{464800}{429} P_{10} P_4
(Q_2) +\frac{17640}{143} P_{12} P_2 (Q_2) +\frac{753600}{143} P_2
P_4 P_2 (Q_8) +\frac{5731360}{143} P_2 P_4 P_2^2 (Q_6)
+\frac{69163904}{429} P_2 P_4 P_2^3 (Q_4) +\frac{103923136}{429}P_2
P_4 P_2^4 (Q_2) +\frac{2735488}{143} P_2 P_4 P_6 (Q_4)
-\frac{1124480}{429} P_2 P_4 P_8 (Q_2) -\frac{883904}{11} P_2 P_4^2
P_2 (Q_4) -\frac{51796640}{429} P_2 P_4^2 P_2^2 (Q_2)
-\frac{86046}{13} P_2 P_6 P_2 (Q_6) -\frac{3803560}{143}P_2 P_6
P_2^2 (Q_4) -\frac{5713556}{143} P_2 P_6 P_2^3 (Q_2)
+\frac{2278296}{143} P_2 P_6 P_4 (Q_4) +\frac{76352}{33} P_2 P_8 P_2
(Q_4) +\frac{1490080}{429} P_2 P_8 P_2^2 (Q_2) -\frac{75200}{39} P_2
P_8 P_4 (Q_2) -\frac{72100}{429} P_2 P_{10} P_2 (Q_2) +43040 P_2^2
P_4 P_2 (Q_6) +\frac{6755200}{39} P_2^2 P_4 P_2^2 (Q_4)
+\frac{10151360}{39}P_2^2 P_4 P_2^3 (Q_2) +\frac{283680}{11} P_2^2
P_4 P_6 (Q_2) -\frac{55694240}{429} P_2^2 P_4^2 P_2 (Q_2)
-\frac{4569640}{143} P_2^2 P_6 P_2 (Q_4) -\frac{6866900}{143} P_2^2
P_6 P_2^2 (Q_2) +\frac{3655360}{143}P_2^2P_6P_4(Q_2)
+\frac{2137760}{429} P_2^2P_8P_2(Q_2)
+\frac{74539904}{429}P_2^3P_4P_2(Q_4)
+\frac{37337920}{143}P_2^3P_4P_2^2(Q_2)
-\frac{6933556}{143}P_2^3P_6P_2(Q_2) +\frac{37394112}{143}
P_2^4P_4P_2(Q_2) -\frac{6013312}{143}P_4P_2P_4(Q_6)
+\frac{45393664}{429} P_4P_2P_4^2(Q_2)
+\frac{4490976}{143}P_4P_2P_6(Q_4) -\frac{160384}{39} P_4P_2P_8(Q_2)
-\frac{5247520}{33}P_4P_2^2P_4(Q_4) +\frac{5653720}{143}
P_4P_2^2P_6(Q_2) -\frac{91195136}{429}P_4P_2^3P_4(Q_2)
+\frac{4402656}{143} P_4P_6P_2(Q_4)
+\frac{6616560}{143}P_4P_6P_2^2(Q_2) -\frac{3521792}{143}
P_4P_6P_4(Q_2) -\frac{192880}{39}P_4P_8P_2(Q_2)
+\frac{15023488}{143} P_4^2P_2P_4(Q_2)
+\frac{6423816}{143}P_6P_2P_4(Q_4) -\frac{1596522}{143}
P_6P_2P_6(Q_2) +\frac{8589120}{143}P_6P_2^2P_4(Q_2)
+\frac{525952}{11} P_6P_4P_2(Q_4)
+\frac{10277440}{143}P_6P_4P_2^2(Q_2) -\frac{3854720}{429}
P_8P_2P_4(Q_2) -\frac{4814320}{429}P_8P_4P_2(Q_2)
-\frac{11501632}{143} P_2P_4P_2P_4(Q_4)
+\frac{2862544}{143}P_2P_4P_2P_6(Q_2)
-\frac{46105600}{429}P_2P_4P_2^2P_4(Q_2)
+\frac{3425184}{143}P_2P_4P_6P_2(Q_2)
+\frac{2535584}{143}P_2P_6P_2P_4(Q_2) +\frac{3042144}{143}
P_2P_6P_4P_2(Q_2) -\frac{4503040}{39}P_2^2P_4P_2P_4(Q_2)
-\frac{20177152}{143}P_4P_2P_4P_2(Q_4)
-\frac{90976640}{429}P_4P_2P_4P_2^2(Q_2)
+\frac{5624584}{143}P_4P_2P_6P_2(Q_2)
-\frac{30378880}{143}P_4P_2^2P_4P_2(Q_2)
+\frac{8582944}{143}P_6P_2P_4P_2(Q_2) -\frac{46084864}{429}
P_2P_4P_2P_4P_2(Q_2) - \frac{2048}{429} i^* \bar{P}_2^7 (\bar{Q}_2)
\).
\end{sloppypar}

\subsection{Tests on round spheres}\label{spheres}

Here we describe some details of a test which confirms the
universality of the displayed formulae for $Q_6$, $Q_8$ on the
spheres $S^n$ of arbitrary even dimension $n$ with the round metric
$h_0$. Similar tests yield the correct values for $Q_{2N}$ for all
$N \le 10$. Basically the same calculations cover the case of
Einstein metrics. This test also illustrates the role of the terms
$i^* \bar{P}_2^k(\bar{Q}_2)$.

On $(S^n,h_0)$, the GJMS-operators are given by the product formulae
\begin{equation}\label{product}
P_{2N} = \prod_{j=\f}^{\f+N-1} (\Delta \!-\! j(n\!-\!1\!-\!j))
\end{equation}
(\cite{comp}, \cite{Beck}, \cite{G-power}). \eqref{product} implies
\begin{equation}\label{p2n-sphere}
P_{2N}(1) = (-1)^N \prod_{j=\f}^{\f+N-1} j(n\!-\!1\!-\!j).
\end{equation}
Using \eqref{q-def}, i.e.,
$$
P_{2N}(1) = (-1)^N (m\!-\!N) Q_{2N}, \quad m = \f,
$$
we find
\begin{equation}\label{q2n-sphere}
Q_{2N} = m \prod_{j=1}^{N-1}(m^2\!-\!j^2).
\end{equation}
These formulae suffice to determine the first $2^{N-1}\!-\!1$ terms
in \eqref{URF}. In order to determine the contributions
$$
i^* \bar{P}_2^{N-1}(\bar{Q}_2),
$$
we note that $\Rho = \frac{1}{2} h_0$, i.e.,
$$
h_r = \left(1-\frac{r^2}{2} \Rho\right)^2 = (1-cr^2)^2 h_0
$$
with $c=\frac{1}{4}$ (by \eqref{terminate}). Hence
$$
\bar{P}_2 = \frac{\partial^2}{\partial r^2} - m r (1\!-\!cr^2)^{-1}
\frac{\partial}{\partial r} - m(m\!-\!2c)(1\!-\!cr^2)^{-1}
$$
on functions which are constant on $M$. Moreover, we have
$$
\bar{Q}_2 = m(1-cr^2)^{-1}
$$
by \eqref{q-flat}.

Now straightforward calculations yield the results
\begin{align*}
i^* \bar{P}_2 (\bar{Q}_2) & = - 2^{-1} m(m\!-\!1)(2 m \!+\! 1), \\
i^* \bar{P}_2^2 (\bar{Q}_2) & = 2^{-2} m(m\!-\!1)(4 m^3 \!-\!
5m\!-\! 6)
\end{align*}
On the other hand, a calculation using \eqref{p2n-sphere} and
\eqref{q2n-sphere} yields
$$
\frac{1}{3} (-5 m^5 + 8m^4 - 5m^3 - 2m^2)
$$
for the sum of the first three terms in the universal formula
\eqref{q6-ex} for $Q_6$. Together with the contribution of $i^*
\bar{P}_2^2 (\bar{Q}_2)$, we obtain
$$
m^5\!-\!5m^3\!+\!4m = m(m^2\!-\!1)(m^2\!-\!4)
$$
which coincides with $Q_6$ by \eqref{q2n-sphere}.

Another calculation using \eqref{p2n-sphere} and \eqref{q2n-sphere}
yields
$$
\frac{1}{5} (-11 m^7 \!+\! 24 m^6 \!-\! 34 m^5 \!+\! 18 m^4 \!+\!
133 m^3 \!-\! 130 m^2)
$$
for the first seven terms in the universal formula for $Q_8$.
Together with the contribution
$$
i^* \bar{P}_2^3 (\bar{Q}_2) = - 2^{-3} m(m\!-\!1)(8 m^5 \!-\! 4 m^4
\!-\! 22 m^3 \!-\! 31 m^2 \!+\! 25 m \!+\! 90)
$$
we find
$$
m^7 \!-\! 14 m^5 \!+\! 49 m^3 \!-\! 36 m =
m(m^2\!-\!1)(m^2\!-\!4)(m^2\!-\!9)
$$
which coincides with $Q_8$ by \eqref{q2n-sphere}.

By \cite{gov-ein}, the product formula \eqref{product} generalizes
in the form
$$
P_{2N}(h) = \prod_{j=\f}^{\f+N-1} \left(\Delta \!-\!
\frac{\tau(h)}{n(n\!-\!1)} j(n\!-\!1\!-\!j)\right)
$$
to Einstein metrics. In particular,
$$
Q_{2N} = \lambda^N m \prod_{j=1}^{N-1} (m^2\!-\!j^2), \quad \lambda
= \frac{\tau}{n(n\!-\!1)}.
$$
Moreover, $\Rho = \frac{1}{2} \lambda h$ and
$$
h_r = \left(1\!-\!\frac{r^2}{2} \Rho \right)^2 =
\left(1\!-\!c\lambda r^2 h \right)^2, \; c = \frac{1}{4}.
$$
Hence (on functions which are constant on $M$),
$$
\bar{P}_2 (h) = \frac{\partial^2}{\partial r^2} - m \lambda r
(1\!-\! c \lambda r^2)^{-1} \frac{\partial}{\partial r} -
m(m\!-\!2c)(1\!-\!c \lambda r^2)^{-1}
$$
and
$$
\bar{Q}_2(h) = m \lambda (1\!-\!c\lambda r^2)^{-1}.
$$
Therefore, for Einstein $h$ with $\tau = n(n\!-\!1)$, the same
calculations as on round spheres, prove \eqref{URF}. For general
scalar curvature, the result follows by rescaling. The assertion is
trivial for $\tau = 0$.

\subsection{The averages $\sigma_{(k,j)}$}\label{aver}

We consider averages of the polynomials $r_I$ over certain sets of
compositions $I$ of the same size $|I|$. We speculate that these
averages can be described in terms of standard interpolation
polynomials.

\begin{defn}[\bf Standard interpolation polynomials] For given integers
$M,N$ such that $N-1 \ge M \ge 0$, let $I_{(M,N)}(x)$ be the
interpolation polynomial of degree $2N-1$ which satisfies
$$
I_{(M,N)}\left(\frac{1}{2}-i\right) = 1, \quad i=0, \dots, N
$$
and
$$
I_{(M,N)} \left(-M-i\right) = 0, \quad i=1,\dots,N-1.
$$
\end{defn}

We use the polynomials $r_I$ for compositions $I$ of size $|I|=j$ to
define the $j$ averages $\sigma_{(k,j)}$, $1 \le k \le j$.

\begin{defn}[\bf Averages] For $j \ge 1$ and $1 \le k \le j$, let
$$
\sigma_{(k,j)}(x) = \sum_{k+|J|=j} r_{(k,J)}(x).
$$
\end{defn}

In particular,
$$
\sigma_{(1,j)} = \sum_{|J|=j-1} r_{(1,J)},
$$
and
\begin{equation}\label{special-sigma}
\begin{split}
\sigma_{(j-2,j)} & = r_{(j-2,1,1)} + r_{(j-2,2)}, \\
\sigma_{(j-1,j)} & = r_{(j-1,1)}, \\
\sigma_{(j,j)} & = r_{(j)}.
\end{split}
\end{equation}

Now we expect that the averages $\sigma_{(k,j)}$ are related to the
interpolation polynomials $I_{(M,N)}$ through the formula
\begin{equation}\label{AI}
\sigma_{(k,j)}(x) = (-2)^{-(j-1)}
\left[\frac{\left(\frac{1}{2}\right)_{k-1}
\left(\frac{1}{2}\!+\!j\right)_{j-k}} {(k\!-\!1)! (j\!-\!k)!}
\right] I_{(j-k,j)}(x).
\end{equation}
In other words, \eqref{AI} states the equalities
\begin{equation}\label{equal}
\sigma_{(k,j)}\left(\frac{1}{2}\right) =
\sigma_{(k,j)}\left(\frac{1}{2}\!-\!1\right) = \dots =
\sigma_{(k,j)}\left(\frac{1}{2}\!-\!j\right),
\end{equation}
claims that this value coincides with
\begin{equation}
(-2)^{-(j-1)} \left[\frac{\left(\frac{1}{2}\right)_{k-1}
\left(\frac{1}{2}\!+\!j\right)_{j-k}} {(k\!-\!1)! (j\!-\!k)!}
\right],
\end{equation}
and asserts that
\begin{equation}\label{zeros}
\sigma_{(k,j)}(-(j-k)-i) = 0 \quad \mbox{for} \quad i=1,\dots,j-1.
\end{equation}

The $j-1$ zeros in \eqref{zeros} are quite remarkable. In fact,
\eqref{zeros} states that $\sigma_{(k,k)} = r_{(k)}$ has zeros in
$x=-1,\dots,-(k-1)$. These are obvious by the definition of
$r_{(k)}$. But for $k<j$, the zeros of $\sigma_{(k,j)}$ in
\eqref{zeros} are {\em not} obvious from the zeros of the individual
terms $r_I$ defining the sum.

Note that the obvious relation
$$
\sum_{|I|=j} r_I(x) = \sum_{k=1}^j \sigma_{(k,j)}(x)
$$
implies
\begin{equation}\label{sum-sigma}
\sum_{k=1}^j \sigma_{(k,j)} \left(\frac{1}{2}\right) = (-1)^{j-1}
\frac{(2j\!-\!1)!!}{j!} 2^{j-1}
\end{equation}
using the conjectural relation \eqref{sum-r}. On the other hand,
\eqref{sum-sigma} would be consequence of the explicit formula
\begin{equation}\label{sigma-value}
\sigma_{(k,j)}\left(\frac{1}{2}\right) = (-2)^{-(j-1)}
\left[\frac{\left(\frac{1}{2}\right)_{k-1}
\left(\frac{1}{2}\!+\!j\right)_{j-k}} {(k\!-\!1)! (j\!-\!k)!}
\right].
\end{equation}
In fact, comparing the coefficients of $x^{j-1}$ on both sides of
the identity
$$
(1-x)^{-\frac{1}{2}} (1-x)^{-(\frac{1}{2}+j)} = (1-x)^{-(1+j)},
$$
we find
$$
\sum_{k=1}^j \frac{\left(\frac{1}{2}\right)_{k-1}
\left(\frac{1}{2}\!+\!j\right)_{j-k}} {(k\!-\!1)! (j\!-\!k)!} =
\frac{(j\!+\!1)_{j-1}}{(j\!-\!1)!} = \frac{(2j\!-\!1)!}{(j\!-\!1)!
j!} = \frac{(2j\!-\!1)!!}{j!}2^{j-1}.
$$
This yields the assertion \eqref{sum-sigma}.

\subsection{The polynomials $r_I$ for compositions of small size}\label{app-r}

In Table \ref{r2} -- Table \ref{r5}, we list the polynomials $r_I$
for all compositions $I$ with $2 \le |I| \le 5$. In each case, we
factorize off the zeros in the negative integers. We recall that
$r_{(1)}=1$.

\begin{table}[p]
\begin{tabular}{c|c}
$I$ & $r_I$ \\[1mm]
\hline & \\[-3mm]
$(1,1)$ & $-\frac{1}{6} (2+x)(3 - 2 x + 4 x^2)$ \\[2mm]
$(2)$ & $\frac{1}{6} (1+x)(-3 + 2 x + 4 x^2)$
\end{tabular}
\bigskip

\caption{\, $r_I$ for compositions $I$ with $|I|=2$}\label{r2}
\end{table}

\begin{table}[p]
\begin{tabular}{c|c}
$I$ & $r_I$ \\[1mm]
\hline & \\[-3mm]
$(1,1,1)$ & $-\frac{1}{60}(2+x)(3+x)(-25+2x+30x^2+48 x^3)$ \\[2mm]
$(1,2)$ & $\frac{1}{30}(1+x)(3+x)(5-12x+6 x^2+16 x^3)$ \\[2mm]
\hline & \\[-3mm]
$(2,1)$ & $\frac{1}{30}(2+x)(3+x)(-5 - 6 x + 30 x^2 + 16 x^3)$ \\[2mm]
$(3)$ & $-\frac{1}{60}(1+x)(2+x)(-15 + 2 x + 42 x^2 + 16x^3)$
\end{tabular}
\bigskip

\caption{\, $r_I$ for compositions $I$ with $|I|=3$}\label{r3}
\end{table}

\begin{table}[p]
\begin{small}
\begin{tabular}{c|c}
$I$ & $r_I$ \\[1mm]
\hline & \\[-3mm]
$(1,1,1,1)$ & $- \frac{1}{2520}(2+x)(3+x)(4+x)(-1155 - 1826x + 5064 x^2 + 6320 x^3 + 2160 x^4)$ \\[2mm]
$(1,1,2)$ & $\frac{1}{252}(1+x)(3+x)(4+x)(-105 - 82 x + 320 x^2 + 416 x^3 + 144 x^4)$ \\[2mm]
$(1,2,1)$ & $\frac{1}{630}(2+x)(3+x)(4+x)(-105 - 136 x + 264 x^2 + 640 x^3 + 240 x^4)$ \\[2mm]
$(1,3)$ & $-\frac{1}{1680}(1+x)(2+x)(4+x)(105 - 254 x - 168 x^2 + 560 x^3 + 240 x^4)$ \\[2mm]
\hline & \\[-3mm]
$(2,1,1)$ & $\frac{1}{252}(2+x)(3+x)(4+x)(-147-146 x+528 x^2 + 608x^3 + 144 x^4)$ \\[2mm]
$(2,2)$ & $-\frac{1}{5040}(1 + x) (3 + x) (4 + x) (-1785 - 2546 x +
7432 x^2 + 9040 x^3 + 2160 x^4)$ \\[2mm]
\hline & \\[-3mm]
$(3,1)$ & $-\frac{1}{560}(2+x)(3+x)(4+x)(-105 - 106 x +424 x^2 + 400 x^3 + 80 x^4)$ \\[2mm]
$(4)$ & $\frac{1}{1008}(1+x)(2+x)(3+x)(-105-50 x+360 x^2 + 272 x^3 +
48x^4)$
\end{tabular}
\end{small}
\bigskip

\caption{\, $r_I$ for compositions $I$ with $|I|=4$}\label{r4}
\end{table}

\begin{table}[p]
\begin{tiny}
\begin{tabular}{c|c}
$I$ & $r_I$ \\[1mm]
\hline & \\[-2mm]
$(1,1,1,1,1)$ & $-\frac{1}{720} (2 + x)(3 + x)(4 + x)(5 + x)
(-1509 -2140 x + 4960 x^2+ 8480 x^3 + 4024 x^4 + 640 x^5)$ \\[2mm]
$(1,1,1,2)$ & $\frac{1}{7560} (1 + x)(3 + x)(4 + x)(5 + x)
(-10143 - 15270 x + 34228 x^2 + 58952 x^3 + 28112 x^4 + 4480 x^5)$ \\[2mm]
$(1,1,2,1)$ & $\frac{1}{2268} (2 + x)(3 + x)(4 + x)(5 + x)
(-2079 - 3000 x + 7000 x^2 + 11680 x^3 + 5600 x^4 + 896 x^5)$ \\[2mm]
$(1,1,3)$ & $-\frac{1}{6048} (1 + x)(2 + x)(4 + x)(5 + x)
(-2079 - 2808 x + 6744 x^2 + 11296 x^3 + 5544 x^4 + 896 x^5)$ \\[2mm]
$(1,2,1,1)$ & $\frac{1}{11340} (2 + x)(3 + x)(4 + x)(5 + x)
(-8127 - 13110 x + 26180 x^2 + 50840 x^3 + 26992 x^4 + 4480 x^5)$ \\[2mm]
$(1,2,2)$ & $-\frac{1}{2160} (1 + x)(3 + x)(4 + x)(5 + x)
(-1197 - 1800 x + 3712 x^2 + 7208 x^3 + 3848 x^4 + 640 x^5)$ \\[2mm]
$(1,3,1)$ & $-\frac{1}{60480} (2 + x)(3 + x)(4 + x)(5 + x)
(-4347 - 9672 x + 9800 x^2 + 38960 x^3 + 25480 x^4 + 4480 x^5)$ \\[2mm]
$(1,4)$ & $\frac{1}{45360} (1 + x)(2 + x)(3 + x)(5 + x)
(945 - 1776 x - 3680 x^2 + 4840 x^3 + 4760 x^4 + 896 x^5)$ \\[2mm]
\hline & \\[-2mm]
$(2,1,1,1)$ & $\frac{1}{7560} (2 + x)(3 + x)(4 + x)(5 + x)
(-14805 - 20172 x + 50960 x^2 + 79280 x^3 + 34048 x^4 + 4480 x^5)$ \\[2mm]
$(2,1,2)$ & $-\frac{1}{11340} (1 + x)(3 + x)(4 + x)(5 + x)
(-14553- 20010 x + 49828 x^2 +78752 x^3 + 33992 x^4 + 4480 x^5)$ \\[2mm]
$(2,2,1)$ & $-\frac{1}{2160} (2 + x)(3 + x)(4 + x)(5 + x)
(-2043 -2808 x + 6920 x^2 + 11120 x^3 + 4840 x^4 + 640 x^5)$ \\[2mm]
$(2,3)$ & $\frac{1}{7560} (1+x)(2+x)(4+x)(5+x)
(-2457-3960 x + 8520 x^2 + 15040x^3 + 6720 x^4 + 896 x^5)$ \\[2mm]
\hline & \\[-2mm]
$(3,1,1)$ & $-\frac{1}{30240} (2 + x)(3 + x)(4 + x)(5 + x)
(-18711- 24000 x + 65240 x^2 + 95120 x^3 + 37576 x^4 + 4480 x^5)$ \\[2mm]
$(3,2)$ & $\frac{1}{7560} (1 + x)(3 + x)(4 + x)(5 + x)
(-3591 - 4866 x + 12716 x^2 + 18904 x^3 + 7504 x^4 + 896 x^5)$ \\[2mm]
\hline & \\[-2mm]
$(4,1)$ & $\frac{1}{45360} (2 + x)(3 + x)(4 + x)(5 + x)
(-4725 - 5736 x + 17080 x^2 + 22480 x^3 + 8120 x^4 + 896 x^5)$ \\[2mm]
$(5)$ & $-\frac{1}{25920} (1+x)(2+x)(3+x)(4+x) (-945 - 888 x + 3320
x^2 + 3760 x^3 + 1240 x^4 + 128 x^5)$
\end{tabular}
\end{tiny}
\bigskip

\caption{\, $r_I$ for compositions $I$ with $|I|=5$}\label{r5}
\end{table}

\subsection{Some values of $r_I$}\label{app-v}

In Table \ref{v2} -- Table \ref{v5}, we list the values of $r_I$ for
$2 \le |I| \le 5$ on the respective sets $\SV(|I|)$ of
half-integers. We write all values as perturbations by $s_I$ of the
respective values at $x=\frac{1}{2}$. From that presentation it is
immediate that the averages $\sigma_{(k,j)}$ are constant on the
respective sets of half-integers, and one can easily read off the
values of $s_I$. In Table \ref{v2} -- Table \ref{v4}, we also
display some values of $r_I$ on half-integers $\not\in \SV(|I|)$.
These influence the values of corresponding polynomials for
compositions of larger size through the multiplicative recursive
relations. In particular,
$$
s_{(2,1)} + s_{(1,2,1)} = 0 \quad \mbox{and} \quad s_{(3)} +
s_{(1,3)} = 0
$$
at $x=-\frac{7}{2}$, and
$$
s_{(3,1)} + s_{(1,3,1)} = 0 \quad \mbox{and} \quad s_{(4)} +
s_{(1,4)} = 0
$$
at $x=-\frac{9}{2}$. These are special cases of $s_{(1,k,1)} +
s_{(k,1)} = 0$ (see \eqref{s-triple}) and $s_{(1,k)} + s_{(k)} = 0$
(see \eqref{s-double}).

Table \ref{vi2} -- Table \ref{vi5} display the values of $r_I$ for
$2 \le |I| \le 5$ on the respective sets of integers in $[-|I|,2]$.
One can easily confirm that the values $r_I(0)$ in Table \ref{vi4}
are determined by the values of $r_I(1)$ in Table \ref{vi3} and
$r_I(2)$ in Table \ref{vi2} according to the relation $r_{(J,k)}(0)
+ r_J(k) r_{(k)}(0) = 0$ (see \eqref{mult-2}). Similarly, the values
$r_I(0)$ in Table \ref{vi5} are determined by the values of $r_I(1)$
in Table \ref{vi4}, $r_I(2)$ in Table \ref{vi3} and $r_I(3)$ in
Table \ref{vi2} (see Section \ref{sec-mult}).

\begin{table}[p]
\begin{tabular}{c|c|c||c|c|c}
$I$ & $-\frac{7}{2}$ & $-\frac{5}{2}$ & $-\frac{3}{2}$ & $-\frac{1}{2}$ & $\frac{1}{2}$ \\[2mm]
\hline & & & & \\[-3mm]
$(1,1)$ & $-\frac{5}{4} + 16$ & $-\frac{5}{4}+4$ & $-\frac{5}{4}$ & $-\frac{5}{4}$ & $-\frac{5}{4}$ \\[2mm]
$(2)$ & $-\frac{1}{4} - 16$ & $-\frac{1}{4}-4$ & $-\frac{1}{4}$ &
$-\frac{1}{4}$ & $-\frac{1}{4}$
\end{tabular}
\bigskip

\caption{\; Values of $r_I$ ($|I|=2$) on
$\frac{1}{2}-\N_0$}\label{v2}
\end{table}
\bigskip

\begin{table}[p]
\begin{tabular}{c|c||c|c|c|c|c}
$I$ & $-\frac{7}{2}$ & $-\frac{5}{2}$ & $-\frac{3}{2}$ &
$-\frac{1}{2}$ &
$\frac{1}{2}$ \\[2mm]
\hline & & & & & \\[-3mm]
$(1,1,1)$ & $\frac{49}{32} + 20$ & $\frac{49}{32} - 4$ & $\frac{49}{32}$ & $\frac{49}{32}$ & $\frac{49}{32}$ \\[2mm]
$(1,2)$ & $\frac{7}{16}-24$ & $\frac{7}{16} + 4$ & $\frac{7}{16}$ & $\frac{7}{16}$ & $\frac{7}{16}$ \\[2mm]
\hline & & & & & \\[-3mm]
$(2,1)$ & $\frac{7}{16}-8$ & $\frac{7}{16}$ & $\frac{7}{16}$ & $\frac{7}{16}$ & $\frac{7}{16}$ \\[2mm]
$(3)$ & $\frac{3}{32} + 12$ & $\frac{3}{32}$ & $\frac{3}{32}$ &
$\frac{3}{32}$ & $\frac{3}{32}$
\end{tabular}
\bigskip

\caption{\; Values of $r_I$ ($|I|=3$) on
$\frac{1}{2}-\N_0$}\label{v3}
\end{table}
\bigskip

\begin{table}[p]
\begin{small}
\begin{tabular}{c|c||c|c|c|c|c}
$I$ & $-\frac{9}{2}$ & $-\frac{7}{2}$ & $-\frac{5}{2}$ &
$-\frac{3}{2}$ & $-\frac{1}{2}$ & $ \frac{1}{2}$ \\[2mm]
\hline & & & & & & \\[-3mm]
$(1,1,1,1)$ & $-\frac{123}{64}+314$ & $-\frac{123}{64}-16$ & $-\frac{123}{64}+5$
& $-\frac{123}{64}$ & $-\frac{123}{64}$ & $-\frac{123}{64}$ \\[2mm]
$(1,1,2)$ & $-\frac{15}{32}-290$ & $-\frac{15}{32}+20$ & $-\frac{15}{32}-5$ & $-\frac{15}{32}$
& $-\frac{15}{32}$ & $-\frac{15}{32}$ \\[2mm]
$(1,2,1)$ & $-\frac{3}{4}-136$ & $-\frac{3}{4} + 8$ & $-\frac{3}{4}$ & $-\frac{3}{4}$
& $-\frac{3}{4}$ & $-\frac{3}{4}$ \\[2mm]
$(1,3)$ & $-\frac{27}{128}+118$ & $-\frac{27}{128}-12$ &
$-\frac{27}{128}$ & $-\frac{27}{128}$ & $-\frac{27}{128}$ & $-\frac{27}{128}$ \\[2mm]
\hline & & & & & & \\[-3mm]
$(2,1,1)$ & $-\frac{15}{32}-110$ & $-\frac{15}{32} + 4$ & -$\frac{15}{32} + 1$ & $-\frac{15}{32}$
& $-\frac{15}{32}$ & $-\frac{15}{32}$ \\[2mm]
$(2,2)$ & $-\frac{39}{128}+116$ & $-\frac{39}{128}-4$ & $-\frac{39}{128}-1$ & $-\frac{39}{128}$
& $-\frac{39}{128}$ & $-\frac{39}{128}$\\[2mm]
\hline & & & & & & \\[-3mm]
$(3,1)$ & $-\frac{27}{128}+18$ & $-\frac{27}{128}$ & $-\frac{27}{128}$ & $-\frac{27}{128}$
& $-\frac{27}{128}$ & $-\frac{27}{128}$ \\[2mm]
$(4)$ & $-\frac{5}{128}-30$ & $-\frac{5}{128}$ & $-\frac{5}{128}$ &
$-\frac{5}{128}$ & $-\frac{5}{128}$ & $-\frac{5}{128}$
\end{tabular}
\end{small}
\bigskip

\caption{\; Values of $r_I$ ($|I|=4$) on
$\frac{1}{2}-\N_0$}\label{v4}
\end{table}
\bigskip

\begin{table}[p]
\begin{small}
\begin{tabular}{c|c|c|c|c|c|c}
$I$ & $-\frac{9}{2}$ & $-\frac{7}{2}$ & $-\frac{5}{2}$ &
$-\frac{3}{2}$ & $-\frac{1}{2}$ & $ \frac{1}{2}$ \\[2mm]
\hline & & & & & & \\[-2mm]
$(1,1,1,1,1)$ & $\frac{1155}{512}-\frac{1025}{4}$ &
$\frac{1155}{512}+\frac{41}{2}$ & $\frac{1155}{512}-\frac{49}{8}$ &
$\frac{1155}{512}$
& $\frac{1155}{512}$ & $\frac{1155}{512}$ \\[2mm]
$(1,1,1,2)$ & $\frac{99}{128}+\frac{941}{4}$ &
$\frac{99}{128}-\frac{51}{2}$ & $\frac{99}{128}+\frac{49}{8}$ &
$\frac{99}{128}$ & $\frac{99}{128}$ & $\frac{99}{128}$ \\[2mm]
$(1,1,2,1)$ & $\frac{55}{64}+110$ & $\frac{55}{64}-10$ &
$\frac{55}{64}$ & $\frac{55}{64}$ & $\frac{55}{64}$ & $\frac{55}{64}$ \\[2mm]
$(1,1,3)$ & $\frac{165}{1024}-95$ & $\frac{165}{1024}+15$ &
$\frac{165}{1024}$ & $\frac{165}{1024}$ & $\frac{165}{1024}$ & $\frac{165}{1024}$ \\[2mm]
$(1,2,1,1)$ & $\frac{55}{64}+\frac{205}{2}$ & $\frac{55}{64}-7$ &
$\frac{55}{64}-\frac{7}{4}$ & $\frac{55}{64}$ & $\frac{55}{64}$ & $\frac{55}{64}$ \\[2mm]
$(1,2,2)$ & $\frac{231}{512}-\frac{217}{2}$ & $\frac{231}{512}+7$ &
$\frac{231}{512}+\frac{7}{4}$ & $\frac{231}{512}$ & $\frac{231}{512}$ & $\frac{231}{512}$ \\[2mm]
$(1,3,1)$ & $\frac{957}{2048}-18$& $\frac{957}{2048}$ &
$\frac{957}{2048}$& $\frac{957}{2048}$ & $\frac{957}{2048}$ & $\frac{957}{2048}$ \\[2mm]
$(1,4)$ & $\frac{55}{512}+30$ & $\frac{55}{512}$ & $\frac{55}{512}$
& $\frac{55}{512}$ & $\frac{55}{512}$ & $\frac{55}{512}$\\[2mm]
\hline & & & & & & \\[-2mm]
$(2,1,1,1)$ & $\frac{99}{128}+\frac{103}{2}$ & $\frac{99}{128}+2$ &
$\frac{99}{128}-\frac{7}{4}$ & $\frac{99}{128}$ & $\frac{99}{128}$ & $\frac{99}{128}$ \\[2mm]
$(2,1,2)$ & $-\frac{11}{128}-\frac{97}{2}$ & $-\frac{11}{128}-3$ &
$-\frac{11}{128}+\frac{7}{4}$ & $-\frac{11}{128}$ & $-\frac{11}{128}$ & $-\frac{11}{128}$ \\[2mm]
$(2,2,1)$ & $\frac{231}{512}-26$ & $\frac{231}{512}-2$ &
$\frac{231}{512}$ & $\frac{231}{512}$ & $\frac{231}{512}$ & $\frac{231}{512}$ \\[2mm]
$(2,3)$ & $\frac{33}{128} + 23$ & $\frac{33}{128} + 3$ &
$\frac{33}{128}$ & $\frac{33}{128}$ & $\frac{33}{128}$ & $\frac{33}{128}$ \\[2mm]
\hline & & & & & & \\[-2mm]
$(3,1,1)$ & $\frac{165}{1024}-\frac{15}{4}$ &
$\frac{165}{1024}-\frac{3}{2}$ & $\frac{165}{1024}-\frac{3}{8}$ & $\frac{165}{1024}$
& $\frac{165}{1024}$ & $\frac{33}{128}$ \\[2mm]
$(3,2)$ & $\frac{33}{128} + \frac{15}{4}$ & $\frac{33}{128} +
\frac{3}{2}$ & $\frac{33}{128} + \frac{3}{8}$ &
$\frac{33}{128}$ & $\frac{33}{128}$ & $\frac{33}{128}$ \\[2mm]
\hline & & & & & & \\[-2mm]
$(4,1)$ & $\frac{55}{512}$ & $\frac{55}{512}$ & $\frac{55}{512}$ & $\frac{55}{512}$ &
$\frac{55}{512}$& $\frac{55}{512}$\\[2mm]
$(5)$ & $\frac{35}{2048}$ & $\frac{35}{2048}$ & $\frac{35}{2048}$ &
$\frac{35}{2048}$& $\frac{35}{2048}$ & $\frac{35}{2048}$
\end{tabular}
\end{small}
\bigskip

\caption{\; Values of $r_I$ ($|I|=5$) on
$\frac{1}{2}-\N_0$}\label{v5}
\end{table}

\begin{table}[p]
\begin{tabular}{c|c|c|c||c|c|c}
$I$ & $-2$ & $-1$ & $0$ & $1$ & $2$ & $3$ \\[2mm]
\hline & & & & & & \\[-3mm]
$(1,1)$ & $0$ & $-\frac{3}{2}$ & $-1$ & $-\frac{5}{2}$ & $-10$ & $-\frac{55}{2}$ \\[2mm]
$(2)$ & $-\frac{3}{2}$ & $0$ & $-\frac{1}{2}$ & $1$ & $\frac{17}{2}$
& $26$
\end{tabular}
\bigskip

\caption{\; The values of $r_I$ ($|I|=2$) on
$\left\{-2,\dots,3\right\}$}\label{vi2}
\end{table}
\bigskip

\begin{table}[p]
\begin{tabular}{c|c|c|c|c||c|c}
$I$ & $-3$ & $-2$ & $-1$ & $0$ & $1$ & $2$ \\[2mm]
\hline & & & & & & \\[-3mm]
$(1,1,1)$ & $0$ & $0$ & $\frac{3}{2}$ & $\frac{5}{2}$ & $-11$ & $-161$ \\[2mm]
$(1,2)$ & $0$ & $\frac{5}{2}$ & $0$ & $\frac{1}{2}$ & $4$ & $\frac{133}{2}$ \\[2mm]
\hline & & & & & & \\[-3mm]
$(2,1)$ & $0$ & $0$ & $1$ & $-1$ & $14$ & $154$ \\[2mm]
$(3)$ & $\frac{5}{2}$ & $0$ & $0$ & $\frac{1}{2}$ & $-\frac{9}{2}$ &
$-57$
\end{tabular}
\bigskip

\caption{\; The values of $r_I$ ($|I|=3$) on $\left\{-3, \dots, 2
\right\}$}\label{vi3}
\end{table}
\bigskip

\begin{table}[p]
\begin{tabular}{c|c|c|c|c|c||c}
$I$ & $-4$ & $-3$ & $-2$ & $-1$ & $0$ & $1$ \\[2mm]
\hline & & & & & & \\[-3mm]
$(1,1,1,1)$ & $0$ & $0$ & $0$ & $-\frac{15}{4}$ & $11$ & $-\frac{503}{2}$ \\[2mm]
$(1,1,2)$ & $0$ & $0$ & $-\frac{5}{2}$ & $0$ & $-5$ & $110$ \\[2mm]
$(1,2,1)$ & $0$ & $0$ & $0$ & $-1$ & $-4$ & $86$ \\[2mm]
$(1,3)$ & $0$ & $-\frac{35}{8}$ & 0 & 0 & $-\frac{1}{2}$ & $-\frac{69}{8}$ \\[2mm]
\hline & & & & & & \\[-3mm]
$(2,1,1)$ & $0$ & $0$ & $0$ & $\frac{3}{2}$ & $-14$ & $235$  \\[2mm]
$(2,2)$ & $0$ & $0$ & $-\frac{15}{8}$ & $0$ & $\frac{17}{4}$ & $-\frac{227}{2}$  \\[2mm]
\hline & & & & & & \\[-3mm]
$(3,1)$ & $0$ & $0$ & $0$ & $-\frac{9}{8}$ & $\frac{9}{2}$ & $-\frac{297}{4}$ \\[2mm]
$(4)$ & $-\frac{35}{8}$ & $0$ & $0$ & $0$ & $-\frac{5}{8}$ &
$\frac{25}{2}$
\end{tabular}
\bigskip

\caption{\; The values of $r_I$ ($|I|=4$) on $\left\{-4, \dots, 1
\right\}$}\label{vi4}
\end{table}

\begin{table}[p]
\begin{small}
\begin{tabular}{c|c|c|c|c|c|c}
$I$ & $-5$ & $-4$ & $-3$ & $-2$ & $-1$ & $0$ \\[2mm]
\hline & & & & & & \\[-2mm]
$(1,1,1,1,1)$& $0$ & $0$ & $0$ & $0$ & $-\frac{33}{2}$ & $\frac{503}{2}$ \\[2mm]
$(1,1,1,2)$ & $0$ & $0$ & $0$ & $\frac{25}{4}$ & $0$ & $-\frac{161}{2}$ \\[2mm]
$(1,1,2,1)$ & $0$ & $0$ & $0$ & $0$ & $10$ & $-110$ \\[2mm]
$(1,1,3)$ & $0$ & $0$ & $\frac{35}{8}$ & $0$ & $0$ & $\frac{55}{4}$  \\[2mm]
$(1,2,1,1)$ & $0$ & $0$ & $0$ & $0$ & $6$ & $-86$ \\[2mm]
$(1,2,2)$ & $0$ & $0$ & $0$ & $\frac{15}{8}$ & $0$ & $\frac{133}{4}$ \\[2mm]
$(1,3,1)$ & $0$ & $0$ & $0$ & $0$ & $\frac{9}{8}$ & $\frac{69}{8}$ \\[2mm]
$(1,4)$ & $0$ & $\frac{63}{8}$ & $0$ & $0$ & $0$ & $\frac{5}{8}$ \\[2mm]
\hline & & & & & & \\[-2mm]
$(2,1,1,1)$ & $0$ & $0$ & $0$ & $0$ & $21$ & $-235$ \\[2mm]
$(2,1,2)$ & $0$ & $0$ & $0$ & $-\frac{5}{2}$ & $0$ & $77$ \\[2mm]
$(2,2,1)$ & $0$ & $0$ & $0$ & $0$ & $-\frac{17}{2}$ & $\frac{227}{2}$ \\[2mm]
$(2,3)$ & $0$ & $0$ & $\frac{7}{2}$ & $0$ & $0$ & $-13$ \\[2mm]
\hline & & & & & & \\[-2mm]
$(3,1,1)$ & $0$ & $0$ & $0$ & $0$ & $-\frac{27}{4}$ & $\frac{297}{4}$\\[2mm]
$(3,2)$ & $0$ & $0$ & $0$ & $\frac{9}{4}$ & $0$ & $-\frac{57}{2}$ \\[2mm]
\hline & & & & & & \\[-2mm]
$(4,1)$& $0$ & $0$ & $0$ & $0$ & $\frac{3}{2}$ & $-\frac{25}{2}$ \\[2mm]
$(5)$ & $\frac{63}{8}$ & $0$ & $0$ & $0$ & $0$ & $\frac{7}{8}$
\end{tabular}
\end{small}
\bigskip

\caption{\; The values of $r_I$ ($|I|=5$) on $\left\{-5, \dots, 0
\right\}$}\label{vi5}
\end{table}

\newpage

\bibliography{conflit2}
\bibliographystyle{plain}

\addcontentsline{toc}{section}{Tables}

\end{document}